\documentclass{amsart}
\usepackage{amssymb, amsmath, amsthm}
\usepackage[dvipsnames]{xcolor}
\usepackage{hyperref}
\usepackage{stmaryrd}
\usepackage[all]{xy}
\usepackage{tikz-cd}
\usepackage{cleveref}
\usepackage{colonequals}

\newtheorem{prop}{Proposition}[section]

\newtheorem{lem}[prop]{Lemma}
\newtheorem{cor}[prop]{Corollary}
\newtheorem{thm}[prop]{Theorem}
\theoremstyle{definition}
\newtheorem{defi}[prop]{Definition}
\theoremstyle{remark}
\newtheorem{examp}[prop]{Example}

\newtheorem{remar}[prop]{Remark}

\newcommand{\stackcite}[1]{\cite[\href{https://stacks.math.columbia.edu/tag/#1}{Tag #1}]{stacks-project}}

\DeclareMathAlphabet{\mathpzc}{OT1}{pzc}{m}{it}

\DeclareMathOperator{\Hom}{Hom}

\DeclareMathOperator{\Ind}{Ind}

\DeclareMathOperator{\Gal}{Gal}

\def\rank{\mathop{\mathrm{ rank}}\nolimits}

\DeclareMathOperator{\Spec}{Spec}

\DeclareMathOperator{\supp}{Supp}

\DeclareMathOperator{\Rep}{Rep}

\DeclareMathOperator{\Ext}{Ext}

\DeclareMathOperator{\Art}{Art}

\DeclareMathOperator{\Set}{Set}
\DeclareMathOperator{\CRing}{CRing}
\DeclareMathOperator{\Ab}{Ab}

\newcommand{\Indu}[3]{\Ind_{#1}^{#2}{#3}}

\newcommand{\Qp}{\mathbb {Q}_p}
\newcommand{\Zp}{\ZZ_p}
\newcommand{\Qpbar}{\overline{\mathbb{Q}}_p}

\newcommand{\Eins}{\mathbf 1}

\newcommand{\FF}{\mathcal F}

\newcommand{\ZZ}{\mathbb Z}

\newcommand{\DD}{\mathfrak D}

\newcommand{\QQ}{\mathbb Q}
\newcommand{\Aa}{\mathfrak A}

\newcommand{\ab}{\mathrm{ab}}
\newcommand{\Fp}{\mathbb F_p}

\newcommand{\mm}{\mathfrak m}

\newcommand{\OO}{\mathcal O}

\DeclareMathOperator{\ad}{ad}

\DeclareMathOperator{\wtimes}{\widehat{\otimes}}

\newcommand{\pp}{\mathfrak p}

\newcommand{\br}[1]{\llbracket #1\rrbracket}

\newcommand{\alg}{\mathrm{alg}}
\newcommand{\cont}{\mathrm{cont}}

\newcommand{\rhobar}{\overline{\rho}}

\newcommand{\Gm}{\mathbb G_m}

\newcommand{\ps}{\mathrm{ps}}

\newcommand{\gen}{\mathrm{gen}}

\DeclareMathOperator{\Spf}{Spf}

\newcommand{\red}{\mathrm{red}}

\DeclareMathOperator{\Lie}{Lie}

\newcommand{\kbar}{\bar{k}}
\newcommand{\rhobarss}{\bar{\rho}^{\mathrm{ss}}}
\newcommand{\PC}{\mathrm{PC}}
\newcommand{\cPC}{\mathrm{cPC}}
\newcommand{\Map}{\mathrm{Map}}
\newcommand{\hyphen}{\text{-}}

\newcommand{\nat}{\mathrm{nat}}
\newcommand{\psibar}{\overline{\psi}}

\newcommand{\tf}{\mathrm{tf}}

\newcommand{\eqto}{\xrightarrow{\cong}}
\newcommand{\Thetabar}{\overline{\Theta}}
\newcommand{\fDN}{\widehat{\mathfrak D}(N)}
\newcommand{\EE}{\mathcal E}

\newcommand{\rhotilde}{\tilde{\rho}}
\newcommand{\Psibar}{\overline{\Psi}}
\newcommand{\wdp}{\wedge, p}
\newcommand{\abp}{\ab,p}
\newcommand{\Zhat}{\widehat{Z}}

\DeclareMathOperator{\Group}{Group}

\title[On local Galois deformation rings: generalised tori]{On local Galois deformation rings: \\ generalised tori}
\author{Vytautas Pa\v{s}k\={u}nas and  Julian Quast }
\date{\today.}
\begin{document}

\begin{abstract} We study deformation theory of mod $p$ Galois representations 
of $p$-adic fields with values in generalised tori, 
such as $L$-groups of (possibly non-split) tori. We show that the corresponding deformation rings are formally smooth over a group algebra of a finite 
abelian $p$-group. We compute their dimension and the set of irreducible 
components. 
\end{abstract}
\maketitle

\tableofcontents

\maketitle

\section{Introduction}\label{sec_intro}
Let $p$ denote any prime number, let $F$ be a finite extension of $\Qp$, and let $\Gamma_F$ denote its absolute Galois group. Let $L$ be another finite extension of $\Qp$ with ring of integers $\OO$, uniformiser $\varpi$ and residue field $k=\OO/\varpi$.

Let $G$ be a smooth affine group scheme over $\OO$, such that its neutral 
component $G^0$ is a torus and the component group $G/G^0$ is a finite 
\'etale group scheme over $\OO$. We call such group schemes \textit{generalised
tori}. We do not make any assumptions regarding $p$ and the component group 
of $G$. 

After replacing $L$ by a finite unramified extension one 
may assume that $G^0$ is a split torus and $G/G^0$ is 
a constant group scheme. We will assume this for the rest of the introduction. 

An important example of a generalised torus that we have in mind is the $L$-group of a torus defined over $F$:
let $H$ be a torus over $F$, which splits over a finite Galois extension 
$E$ over $F$, then the $L$-group of $H$ is ${}^L H= \hat{H}\rtimes \underline{\Gal(E/F)}$, where $\hat{H}$ is the split $\OO$-torus, such that the character lattice of $\hat{H}$ is equal to the cocharacter lattice 
of $H_E$ and $\underline{\Gal(E/F)}$ is the constant group scheme 
associated to $\Gal(E/F)$. In this example, the surjection 
$G\twoheadrightarrow G/G^0$ has a section of group schemes. We do not assume 
this in general.

We fix a continuous representation $\rhobar:\Gamma_F\rightarrow G(k)$ and denote by $D^{\square}_{\rhobar}: \mathfrak A_{\OO}\rightarrow \Set$ the functor from the category $\mathfrak A_{\OO}$ of local artinian $\OO$-algebras with residue field $k$ to the category of sets, such that for $(A,\mm_A)\in \mathfrak A_{\OO}$, $D^{\square}_{\rhobar}(A)$ is the set of continuous representations 
$\rho_A: \Gamma_F\rightarrow G(A)$, such that 
$$\rho_A(\gamma) \equiv \rhobar(\gamma) \pmod{\mm_A}, \quad \forall\gamma\in \Gamma_F.$$
The 
functor $D^{\square}_{\rhobar}$ of framed deformations of $\rhobar$ is 
pro-represented by a complete local 
noetherian $\OO$-algebra $R^{\square}_{\rhobar}$ with residue field~$k$.

\begin{thm}\label{intro_A} There is a finite extension $L'$ of $L$ 
with ring of integers $\OO'$ and a 
continuous representation $\rho:\Gamma_F \rightarrow G(\OO')$
lifting $\rhobar$.
\end{thm} 

\begin{remar}\label{intro_rmk} If $G\cong G^0\rtimes (G/G^0)$ then a lift
as in Theorem \ref{intro_A} can be constructed over $\OO$ using 
the Teichm\"uller lift to define a section of $G^0(\OO)\twoheadrightarrow 
G^0(k)$, which then induces a section $\sigma: G(k)\rightarrow G(\OO)$. 
Then $\sigma \circ \rhobar$ is the required lift. 
In particular, in this case Theorem \ref{intro_A} holds
if we replace $\Gamma_F$ with  any profinite group. 
However, it seems non-trivial to prove \Cref{intro_A} in general and our argument
uses that the Euler--Poincar\'e characteristic formula holds for $\Gamma_F$.
\end{remar}

Let $\Gamma_E$ be the kernel of the composition 
$\Gamma_F \overset{\rhobar}{\longrightarrow} G(k)\rightarrow (G/G^0)(k)$
and let $\Delta$ be the image of this map. We identify $\Delta=\Gal(E/F)$.
Let $M$ be the character lattice of $G^0$. The action of $G$ on $G^0$ by conjugation induces an action of $\Delta$ on $M$. Let $\Gamma_E^{\ab, p}$ 
be the maximal abelian pro-$p$ quotient of $\Gamma_E$. Below 
we will consider the diagonal action of $\Delta$ on $\Gamma_E^{\ab,p}\otimes M$
and $\mu_{p^{\infty}}(E)\otimes M$, where $\mu_{p^{\infty}}(E)$ is 
the subgroup of $p$-power roots of unity in $E$. 

\begin{thm}\label{intro_B} Assume that Theorem \ref{intro_A} holds
with $\OO=\OO'$. Then 
$$R^{\square}_{\rhobar} \cong 
\OO[(\mu_{p^{\infty}}(E)\otimes M)^{\Delta}]\br{x_1, \ldots, x_m},$$
where $m= \rank_{\ZZ} M \cdot ([F:\Qp]+1)$. 
\end{thm} 

Let $\Thetabar$ be the $G$-pseudocharacter associated to $\rhobar$ and 
let $D^{\ps}_{\Thetabar}: \Aa_{\OO} \rightarrow \Set$ be the deformation functor
such that $D^{\ps}_{\Thetabar}(A)$ is the set of 
continuous $A$-valued $G$-pseudocharacters deforming $\Thetabar$. These 
notions are reviewed in \Cref{Laf}. The functor $D^{\ps}_{\Thetabar}$ is 
pro-represented by a complete local 
noetherian $\OO$-algebra $R^{\ps}_{\Thetabar}$ with residue field~$k$
and we denote the universal deformation by $\Theta^u$.
Sending a deformation of $\rhobar$ to its $G$-pseudocharacter induces 
a natural transformation $D^{\square}_{\rhobar} \rightarrow D^{\ps}_{\Thetabar}$, 
and hence a map of local $\OO$-algebras
$R^{\ps}_{\Thetabar}\rightarrow R^{\square}_{\rhobar}$. 
If $G$ is a torus or more generally, when $G$ is commutative,
then representations and $G$-pseudocharacters coincide and this map is an isomorphism.  

\begin{thm}\label{intro_C} The map $R^{\ps}_{\Thetabar} \rightarrow 
R^{\square}_{\rhobar}$ is formally smooth. Moreover, if Theorem 
\ref{intro_A} holds with $\OO'=\OO$ then 
$$ R^{\ps}_{\Thetabar}\cong \OO\br{(\Gamma_E^{\abp}\otimes M)^{\Delta}}\cong 
\OO[(\mu_{p^{\infty}}(E)\otimes M)^{\Delta}]\br{x_1, \ldots, x_r},$$
where $r=\rank_{\ZZ} M \cdot [F:\Qp]+ \rank_{\ZZ} M_{\Delta}$. 
\end{thm}

We let $X^{\gen}_{G, \Thetabar}: R^{\ps}_{\Thetabar}\text{-}\alg\rightarrow \Set$
be the functor such that $X^{\gen}_{G, \Thetabar}(A)$ is the set of representations
$\rho: \Gamma_F\rightarrow G(A)$ such that its 
$G$-pseudocharacter $\Theta_{\rho}$ satisfies $\Theta_{\rho} =\Theta^u \otimes_{R^{\ps}_{\Thetabar}} A$. This functor is representable by a 
finite type $R^{\ps}_{\Thetabar}$-algebra $A^{\gen}_{G, \Thetabar}$.

\begin{thm}\label{intro_D} The map 
$R^{\ps}_{\Thetabar}\rightarrow A^{\gen}_{G, \Thetabar}$ is smooth. 
Moreover, if Theorem \ref{intro_A} holds with $\OO'=\OO$ then 
$$ A^{\gen}_{G, \Thetabar}\cong R^{\ps}_{\Thetabar}[t_1^{\pm 1}, \ldots, t_s^{\pm 1}],$$
where $s= \rank_{\ZZ} M -\rank_{\ZZ} M_{\Delta}$. 
\end{thm}

We note that $\mu_{p^{\infty}}(E)$ is a finite $p$-group. This implies that 
$\mu:=(\mu_{p^{\infty}}(E)\otimes M)^{\Delta}$ is also a finite $p$-group 
and we denote its order by $p^m$. We assume further that $\OO$ contains 
all $p^m$-th roots of $1$ and let $\mathrm X(\mu)$ be the group 
of characters $\chi: \mu\rightarrow \OO^{\times}$. It then follows from Theorems \ref{intro_B}, \ref{intro_C}, \ref{intro_D} that 
the sets of irreducible components of $R^{\square}_{\rhobar}$, 
$R^{\ps}_{\Thetabar}$
and $A^{\gen}_{G, \Thetabar}$, respectively,  are 
in $\mathrm X(\mu)$-equivariant bijection with $\mathrm X(\mu)$. 
Moreover, the irreducible components and their special fibres are regular.  The identification of the set 
of components with $\mathrm X(\mu)$ is non-canonical in general, as one has to 
distinguish one component, which corresponds to the trivial 
character. However, we explain in \Cref{sec_irr_comp} that there 
is a canonical action of $\mathrm X(\mu)$ on the set of irreducible 
components, which is faithful and transitive.
\begin{remar} Let us point out that if $G\cong G^0\rtimes G/G^0$ then 
the lift constructed in \Cref{intro_rmk} is canonical, as it is minimally ramified. This distinguishes the irreducible component that it lies on. Thus in 
this case there is a canonical $\mathrm X(\mu)$-equivariant bijection between 
$\mathrm X(\mu)$ and the set of irreducible components. 
\end{remar}

The above theorems are used in an essential way in a companion paper \cite{defG},
where we study deformations 
of Galois representations with values in generalised reductive 
$\OO$-group schemes $G$, which means that the neutral  component $G^0$ is 
reductive and the component group $G/G^0$ is finite \'etale. If we 
let $G'$ be the derived subgroup scheme of $G^0$, then 
$G^0/G'$ is a torus and $G/G'$ is a generalised torus as considered 
in this paper. If $\rhobar: \Gamma_F \rightarrow G(k)$ is a continuous 
representation then we show in \cite[Proposition 4.20]{defG} that its deformation ring $R^{\square}_{\rhobar}$ 
can be presented over the deformation ring of 
$\varphi \circ \rhobar: \Gamma_F \rightarrow (G/G')(k)$, where 
$\varphi: G\rightarrow G/G'$ is the quotient map. This allows 
us to split up the arguments into ``torus part", carried out in this paper, 
and ``semisimple part" carried out in \cite{defG}.

If $\mu$ is any finite abelian $p$-group then $\mu\cong \prod_{i=1}^l \ZZ/p^{e_i}$ and 
hence
$$\OO[\mu]\cong \bigotimes_{i=1}^l \frac{\OO[z_i]}{( (1+z_i)^{p^{e_i}}-1)}$$
is complete intersection. It then follows from Theorems \ref{intro_B}, 
\ref{intro_C}, \ref{intro_D} that the rings $R^{\square}_{\rhobar}$, 
$R^{\ps}_{\Thetabar}$ and $A^{\gen}_{G, \Thetabar}$ are locally complete 
intersections. To relate their dimensions to the dimensions of the deformation
rings appearing in \cite{defG} we note that $\rank_{\ZZ} M= \dim G_k$ and,
if $\Delta=(G/G^0)(\kbar)$ (which we may assume without changing the 
functors $D^{\square}_{\rhobar}$, $D^{\ps}_{\Thetabar}$ and $X^{\gen}_{G,\Thetabar}$), then $\rank_{\ZZ} M_{\Delta} = \dim Z(G)_k$.
The scheme $X^{\gen}_{G,\Thetabar}$ 
coincides with the scheme denoted by $X^{\gen}_{G, \rhobarss}$ in \cite{defG}, when 
$G$ is a generalised torus. In \cite{defG} in the definition of $X^{\gen}_{G, \rhobarss}$ an additional condition 
is imposed on representations, called $R^{\ps}_{\Thetabar}$-condensed, and
we show in Lemma \ref{defG2}, that when $G$ is a generalised torus, this condition 
holds automatically. 

\subsection{Arithmetic setting} Let us first recall how one proves 
the result when $G=\Gm$. The first step is to observe that 
Theorem \ref{intro_A} holds with $\OO'=\OO$ for example 
using the Teichm\"uller lift $\sigma: k^{\times}\rightarrow \OO^{\times}$ 
and letting $\rho_0=\sigma \circ \rhobar$. The second step 
is to observe that the map $\rho\mapsto \rho\rho^{-1}_0$ induces
a bijection between $D^{\square}_{\rhobar} \rightarrow D^{\square}_{\Eins}$, 
where $\Eins$ is the trivial representation. One may identify 
$D^{\square}_{\Eins}(A)$ with the set of continuous group homomorphisms
$\Gamma_F \rightarrow 1+\mm_A$ and since the target is an abelian 
$p$-group the functor is representable by $\OO\br{\Gamma_F^{\abp}}$.
An interesting feature of this argument is that $D^{\square}_{\Eins}$ 
takes values in $\Ab$ as opposed to $\Set$ and the map 
$$D_{\Eins}^{\square}(A) \times D_{\rhobar}^{\square}(A)\rightarrow D^{\square}_{\rhobar}(A), \quad (\Phi, \rho)\mapsto \Phi \rho$$
defines an action of an abelian group $D_{\Eins}^{\square}(A)$ on the 
set $D^{\square}_{\rhobar}(A)$. Moreover,  if the set is non-empty then this 
action is faithful and transitive. 

If $G$ is connected then $G=\DD(M)$, where $M$ is the character lattice, 
and so $G(A)=\Hom(M, A^{\times})$. The same argument goes through and we only 
need to observe that 
$$\Hom^{\cont}_{\mathrm{Group}}(\Gamma_F, G(A))\cong 
\Hom^{\cont}_{\mathrm{Group}}(\Gamma_F^{\ab}\otimes M, A^{\times}),$$
which implies that $D^{\square}_{\rhobar}$ is representable by 
$\OO\br{\Gamma_F^{\abp}\otimes M}$. 

Let us consider the general case now. The first issue, alluded to
in \Cref{intro_rmk}, is that it is not clear anymore that such a lift exists. 
However, let us postpone the explanation how we deal with this problem 
until \Cref{sec_intro_lift} and let us give 
ourselves a lift $\rho_0: \Gamma_F \rightarrow G(\OO)$ of $\rhobar$. 
Then the map $\rho \mapsto [\gamma \mapsto \rho(\gamma) \rho_0(\gamma)^{-1}]$ considered above induces a bijection between $D^{\square}_{\rhobar}(A)$
and the set $\Zhat^1(A)$ of continuous $1$-cocycles $\Phi: \Gamma_F \rightarrow \Hom(M, 1+\mm_A)$. Further, one may interpret $\Zhat^1$ as 
$D^{\square}_{\rhobar}$ for the group $G=\DD(M)\rtimes \underline{\Delta}$ 
and $\rhobar: \Gamma_F \rightarrow G(k)$ given by $\rhobar(\gamma)= (1, \pi(\gamma))$, where $\pi: \Gamma_F\rightarrow \Delta$ is the quotient map. 
The functor $\Zhat^1$ is the analogue of $D^{\square}_{\Eins}$ and we have
to find the ring pro-representing the functor $\Zhat^1$. 

Before we describe our solution, let us point out one dead end. If 
$\Phi\in \Zhat^1(A)$ then its restriction to $\Gamma_E$ is 
just a continuous homomorphism 
$$\Phi|_{\Gamma_E}: \Gamma_E \rightarrow \Hom(M, 1+\mm_A),$$
as $\Gamma_E$ acts trivially on $M$. 
This induces a continuous homomorphism 
$$\varphi: \Gamma_E^{\abp}\otimes M \rightarrow 
1+\mm_A,$$ and the fact that the homomorphism is obtained 
by a restriction of a cocycle to $\Gamma_E$ implies that 
$\varphi$ factors through $(\Gamma_E^{\abp}\otimes M)_{\Delta}\rightarrow 1+\mm_A$. However, in Theorem \ref{intro_B} we have $\Delta$-invariants and
not $\Delta$-coinvariants appearing. Tate group cohomology gives us an exact 
sequence
$$ 0\rightarrow \widehat{H}^{-1}(\Delta, \Gamma_E^{\abp}\otimes M)\rightarrow 
    (\Gamma_E^{\abp}\otimes M)_{\Delta} \rightarrow (\Gamma_E^{\abp}\otimes M)^{\Delta}\rightarrow 
    \widehat{H}^0(\Delta, \Gamma_E^{\abp}\otimes M)\rightarrow 0. $$
One may show that the Tate group cohomology groups $\widehat{H}^i(\Delta, 
\Gamma_E^{\abp}\otimes M)$ are finite for all $i\in \ZZ$, but there 
is no reason why they should vanish, unless $p$ does not divide $|\Delta|$. 
We conclude that the restriction to $\Gamma_E$ will lose information in general. 

Our solution is motivated by Langlands' paper \cite{langlands_tori} 
on his correspondence for tori, where he shows that the universal coefficient
theorem induces an isomorphism
\begin{equation}\label{intro_univ_coeff}
H^1(W_{E/F}, \Hom(M, \mathbb C^{\times}))\cong \Hom (H_1(W_{E/F}, M), \mathbb C^{\times}))
\end{equation}
and there are natural isomorphisms 
\begin{equation}\label{intro_fundamental}
H_1(W_{E/F}, M) \cong H_1(E^{\times}, M)^{\Delta} \cong (E^{\times}\otimes M)^{\Delta},
\end{equation}
where the Weil group $W_{E/F}$ fits into a short exact sequence
\begin{equation}\label{intro_weil}
0\rightarrow E^{\times} \rightarrow W_{E/F}\rightarrow \Delta\rightarrow 0,
\end{equation}
corresponding to the fundamental class $[u_{E/F}]\in H^2(\Delta, E^{\times})$. 
Since every continuous representation $\Gamma_F \rightarrow G(A)$ with image in $(G/G^0)(A)$ equal to $\Delta$ factors 
through the profinite completion of $W_{E/F}$, this motivates us 
to study the functor 
$$\CRing \rightarrow \Ab,\quad  A\mapsto Z^1(W_{E/F}, G(A))$$
without imposing any continuity conditions on the cocycles and allowing 
any commutative ring $A$. We determine the ring representing this functor 
and show that the ring representing $\Zhat^1$ arises as a completion of it with respect to the maximal ideal corresponding to the trivial cocycle. 

\subsection{Abstract setting}
We carry out this study in \Cref{sec_Z1} in the following abstract setting:
let $\Gamma_1$ be an abstract group, $\Gamma_2$ a normal subgroup 
of $\Gamma_1$ of finite index and let $\Delta:=\Gamma_1/\Gamma_2$. 
The exact sequence
\begin{equation}\label{EE0}
 0\rightarrow \Gamma_2^{\ab}\rightarrow \Gamma_1/[\Gamma_2, \Gamma_2]\rightarrow \Delta \rightarrow 0
\end{equation} 
defines a class in $H^2(\Delta, \Gamma_2^{\ab})$. Its image 
in $\Ext^1_{\ZZ[\Delta]}(I_{\Delta}, \Gamma_2^{\ab})$, where
$I_{\Delta}$ is the augmentation ideal in the group ring $\ZZ[\Delta]$
defines an extension of $\ZZ[\Delta]$-modules
\begin{equation}\label{EE}
0\rightarrow \Gamma_2^{\ab}\rightarrow \mathcal E\rightarrow I_{\Delta}\rightarrow 0.
\end{equation}

\begin{thm}\label{intro_E} The functor 
$$\CRing \rightarrow \Ab, \quad A\mapsto Z^1(\Gamma_1, \Hom(M, A^{\times}))$$
is represented by the group algebra $\ZZ[(\mathcal E\otimes M)_{\Delta}]$. 

The fpqc sheafification of the functor 
$$ \CRing \rightarrow \Ab, \quad A\mapsto H^1(\Gamma_1, \Hom(M, A^{\times}))$$
is represented by the group algebra $\ZZ[H_1(\Gamma_1, M)]$.
\end{thm}

We prove \Cref{intro_E} by first replacing $M$ with $\ZZ[\Delta]\otimes M$. In this 
case 
$$\Hom(\ZZ[\Delta]\otimes M, A^{\times})\cong \Ind^{\Gamma_1}_{\Gamma_2} \Hom(M, A^{\times})$$
and we can use Shapiro's lemma to identify $H^1$ with 
$$ \Hom(\Gamma_2, \Hom(M, A^{\times}))\cong \Hom(\Gamma_2^{\ab}\otimes M, A^{\times})$$
Further, after 
choosing a representative $\overline{c}$ for each coset $c\in \Gamma_1/\Gamma_2$, 
we construct in \Cref{sec_shap} a homomorphism of abelian groups,
$$ H^1(\Gamma_1,\Ind_{\Gamma_2}^{\Gamma_1} V) \rightarrow Z^1(\Gamma_1, \Ind_{\Gamma_2}^{\Gamma_1} V),$$ 
for any $\Gamma_2$-module $V$, which is functorial in $V$ and becomes the identity, 
when composed with the natural quotient map. These ingredients allow us to 
show that the functor
\begin{equation}\label{funct_funct}
\Ab\rightarrow \Ab, \quad V\mapsto Z^1(\Gamma_1, \Hom(\ZZ[\Delta]\otimes M, V))
\end{equation}
is representable by $(\Gamma_2^{\ab}\oplus I_{\Delta})\otimes M$. We then define an action 
of $\Delta$ on $\ZZ[\Delta]\otimes M$, which commutes with the action 
of $\Gamma_1$, such that 
$$ Z^1(\Gamma_1, \Hom(M, V)) \cong Z^1(\Gamma_1, \Hom(\ZZ[\Delta]\otimes M, V))^{\Delta}.$$
To prove \Cref{intro_E} we verify that the action of $\Delta$ 
on  $(\Gamma_2^{\ab}\oplus I_{\Delta})\otimes M$ induced by 
its  action on the functor in \eqref{funct_funct} is isomorphic 
to $\mathcal E\otimes M$. This calculation is carried out in 
\Cref{sec_Z1}.
The last part of 
Theorem \ref{intro_E} follows by observing that \eqref{intro_univ_coeff}
remains an isomorphism if $\mathbb C^{\times}$ is replaced by any divisible 
group and for any $A$ there is a faithfully flat map $A\rightarrow B$ 
such that $B^{\times}$ is divisible.

Let $G$ be a generalised torus over some base ring $\OO$. Let 
$\PC^{\Gamma_1}_{G}: \OO\text{-}\alg\rightarrow \Set$ be the 
functor such that $\PC^{\Gamma_1}_{G}(A)$ is the set of $A$-valued $G$-pseudocharacters of $\Gamma_1$. Let 
$\Rep^{\Gamma_1}_G: \OO\text{-}\alg \rightarrow \Set$ be 
the functor such that $\Rep^{\Gamma_1}_G(A)$ is the set 
of all representations 
$\rho:\Gamma_1 \rightarrow G(A)$. 
\begin{thm}\label{intro_F} Mapping a representation 
to its $G$-pseudocharacter induces an isomorphism of 
schemes 
$$ \Rep^{\Gamma_1}_{G}\sslash G^0\overset{\cong}{\longrightarrow} \PC^{\Gamma_1}_{G}. $$
\end{thm}
The proof of \Cref{intro_F} follows the arguments 
of Emerson--Morel \cite{emerson2023comparison}, where
they prove an analogous result, when $G$ is a connected 
reductive group over a field of characteristic $0$. Since in our case 
$G^0$ is a torus, it is linearly reductive over $\OO$, and 
their arguments carry over integrally. 

Let us assume that $\Rep^{\Gamma_1}_G(\OO)$ is non-empty, 
$G^0$ is split over $\OO$ and $G/G^0$ is a constant group 
scheme. Let $\rho_0: \Gamma_1\rightarrow G(\OO)$ be a 
representation and let $\Rep^{\Gamma_1}_{G, \pi}$ 
be the subfunctor of $\Rep^{\Gamma_1}_{G}$ consisting 
of representations $\rho$ such that $\Pi\circ \rho= \Pi\circ\rho_0$, where $\Pi: G \rightarrow G/G^0$ is the quotient map. One 
may similarly define $\PC^{\Gamma_1}_{G, \pi}$. We show that the map
$$ Z^1(\Gamma_1, G^0(A))\rightarrow \Rep^{\Gamma_1}_{G,\pi}(A), \quad \Phi\mapsto \Phi \rho_0$$
is bijective and use Theorem \ref{intro_E} to prove:

\begin{thm}\label{intro_G} The functors $\Rep^{\Gamma_1}_{G, \pi}$ and 
$\PC^{\Gamma_1}_{G, \pi}$ are represented by the group 
algebras $\OO[(\mathcal E \otimes M)_{\Delta}]$ and $\OO[H_1(\Gamma_1, M)]$,
respectively.
\end{thm}
If $\Gamma_1=W_{E/F}$ then $H_1(\Gamma_1, M)\cong (E^{\times}\otimes M)^{\Delta}$
by \eqref{intro_fundamental}. To prove \Cref{intro_C} we 
relate the completions of the local rings of $\Rep^{\Gamma_1}_{G, \pi}$ 
and $\PC^{\Gamma_1}_{G, \pi}$ to deformation rings parameterising 
continuous deformations and $G$-pseudocharacters of the profinite completion 
$\widehat{\Gamma}_1$ of $\Gamma_1$. These arguments are carried out in 
\Cref{sec_profinite}.

We expect that the abstract setting will also be useful studying deformations
of global Galois groups.

\subsection{Producing a lift}\label{sec_intro_lift}
Let us go back to the problem of exhibiting a lift of $\rhobar: \Gamma_F\rightarrow G(k)$ to characteristic zero as claimed in  
\Cref{intro_A}. Using $\rhobar$ instead of $\rho_0$, 
we may identify the restriction of $D^{\square}_{\rhobar}$ to $\Aa_k$ 
with the restriction of $\Zhat^1$ to $\Aa_k$. This allows 
us to conclude that 
$$R^{\square}_{\rhobar}/\varpi \cong k\br{((\mathcal E\otimes M)_{\Delta})^{\wdp}},$$
where $\mathcal E$ is defined by \eqref{EE}, where \eqref{EE0} is 
equal to \eqref{intro_weil}, so that  $\Gamma_1=W_{E/F}$ and 
$\Gamma_2=E^{\times}$ and $(-)^{\wdp}$ denotes the pro-$p$ completion. 

In \Cref{sec_rank} we compute 
$$\rank_{\Zp} ((\mathcal E\otimes M)_{\Delta})^{\wdp}= \rank_{\ZZ} M \cdot ([F:\Qp]+1),$$ which is then also equal to $\dim R^{\square}_{\rhobar}/\varpi$.  On the other hand, using Mazur's obstruction theory and the Euler--Poincar\'e characteristic formula we show that 
$$ R^{\square}_{\rhobar}\cong \frac{\OO\br{x_1, \ldots, x_r}}{(f_1, \ldots, f_s)},$$
where $r-s= \dim_k \Lie G_k\cdot ([F:\Qp]+1)$. Since $\rank_{\ZZ} M = \dim_k \Lie G_k$, a little commutative algebra
implies that $\varpi, f_1, \ldots, f_s$ is a regular sequence in $\OO\br{x_1,\ldots, x_r}$ and hence $R^{\square}_{\rhobar}$ is $\OO$-torsion 
free. A closed point of the generic fibre gives us the required lift. 

\subsection{Overview by section}
In \Cref{gen_tori} we define what a generalised torus is.
In \Cref{sec_shap} we establish an explicit version of Shapiro's lemma. 
In \Cref{sec_Z1} we compute  
the ring representing the functor $A\mapsto Z^1(\Gamma_1, \Hom(M, A^{\times}))$.
This section is the technical heart of the paper.
In \Cref{sec_adm_rep} we relate the functor  $\Rep^{\Gamma_1}_{G, \pi}$ 
to the functor $Z^1(\Gamma_1, G^0(-))$. In \Cref{Laf} we introduce 
Lafforgue's $G$-pseudocharacters and study the functor $\PC^{\Gamma_1}_{G, \pi}$.
In \Cref{sec_profinite} we transfer the results about abstract representations
and $G$-pseudocharacters of $\Gamma_1$ to continuous representations 
and continuous $G$-pseudocharacters of its profinite completion $\widehat{\Gamma}_1$ and prove versions of Theorems
\ref{intro_B}, \ref{intro_C}, \ref{intro_D} with $\Gamma_F$ replaced by 
$\widehat{\Gamma}_1$. In \Cref{sec_rank} we compute $\Zp$-ranks of 
certain pro-$p$ completions  appearing naturally in the arithmetic setting.
We show in Lemma \ref{duck_duck} that the pro-$p$ completion 
of $(E^{\times}\otimes M)^{\Delta}$ is isomorphic to $(\Gamma_E^{\abp}\otimes M)^{\Delta}$. In \Cref{sec_gal_def} we use the results proved in the previous
sections to prove Theorems \ref{intro_A}, \ref{intro_B}, \ref{intro_C}, \ref{intro_D}.

\subsection{Notation} Let $p$ denote any prime number. Let $F$ be a finite extension of $\Qp$. We fix an algebraic closure $\overline{F}$ of $F$. Let $\Gamma_F\colonequals\Gal(\overline{F}/F)$ be the absolute Galois group of $F$. Let $L$ be another finite extension of $\Qp$ with ring of integers $\OO$, uniformiser $\varpi$ and residue field $k$. However, $\OO$ is allowed to be an arbitrary commutative ring such that $\Spec \OO$ is connected in Sections \ref{sec_adm_rep}
and \ref{Laf}.

Let $\Aa_{\OO}$ be the category of local artinian $\OO$-algebras 
with residue field $k$. We let $\widehat{\Aa}_{\OO}$ be the category of
pro-objects of $\Aa_{\OO}$. Concretely, one may identify $\widehat{\Aa}_{\OO}$
with the category of pseudo-compact local $\OO$-algebras with residue field 
$k$. Let $\Aa_k$ be the full subcategory of $\Aa_{\OO}$ consisting of objects 
killed by $\varpi$. Then $\widehat{\Aa}_k$ is the full subcategory of 
$\widehat{\Aa}_{\OO}$ consisting of objects killed by $\varpi$, which can be identified with 
the category of local pseudo-compact $k$-algebras with residue field $k$.

The groups denoted by $\Gamma_1$ and $\Gamma_2$ are abstract groups with 
no topology. We will denote by $\widehat{\Gamma}_1$ the profinite completion 
of $\Gamma_1$. If $\mathcal A$ is an abelian group then we will denote its 
pro-$p$ completion by $\mathcal A^{\wdp}$.


If $X$ is a scheme then we denote the ring of its global sections by $\OO(X)$.
If $\Gamma$ is a group then we denote by $\OO[\Gamma]$ its group algebra over a ring $\OO$. 
\subsection{Acknowledgements} VP would like to thank Timo Richarz for 
a helpful discussion concerning the Langlands correspondence for tori. The authors would like to
thank Toby Gee and James Newton for their comments on a draft version of this paper. The authors thank the anonymous referee for their helpful comments that improved the quality of the manuscript.

Parts of the paper were written during the research stay of VP at 
the Hausdorff Research Institute for Mathematics in Bonn 
for the Trimester Program \textit{The Arithmetic of the Langlands Program}.
VP would like to thank the organisers 
Frank Calegari, Ana Caraiani, Laurent Fargues and  Peter Scholze
for the invitation and a stimulating research environment. 
The research stay of VP was funded by the Deutsche Forschungsgemeinschaft (DFG, German Research Foundation) under Germany's Excellence Strategy – EXC-2047/1 – 390685813. 

The research of JQ was funded by the Deutsche Forschungsgemeinschaft (DFG, German Research Foundation) – project number 517234220.

\section{Generalised tori}\label{gen_tori}

\begin{defi}\label{def_gen_torus} Let $S$ be a scheme. A \emph{generalised torus} is a smooth affine $S$-group scheme $G$, such that the geometric fibres of $G^0$ are tori and $G/G^0 \to S$ is finite.
\end{defi}

\begin{remar}\label{rem_fin_etale}
    It follows from \cite[Proposition 3.1.3]{bcnrd}, that for a generalised torus $G$ the quotient $G/G^0$ is étale. Sean Cotner has shown in \cite{cotner} that if $S$ is locally noetherian then for a smooth affine $S$-group scheme $G$ the identity component $G^0$ has reductive geometric fibers and $G/G^0$ is finite étale if and only if $G$ is geometrically reductive in the sense of \cite[Definition 9.1.1]{alper}. In particular, this holds over $S=\Spec \OO$. 
\end{remar}

\begin{remar}\label{right_exactness_G0}
    When $G$ is a generalised torus over $\OO$, then $G^0$ is linearly reductive, i.e.~taking $G^0$-invariants is an exact functor on the category of {$\OO(G^0)$-comodules}.
    It follows, that if $A$ is a commutative $\OO$-algebra with trivial $G^0$-action, $M$ is an $A$-module with $G^0$-action and $N$ is an $A$-module with trivial $G^0$-action, then $${(M \otimes_A N)^{G^0} = M^{G^0} \otimes_A N}.$$
\end{remar}

If $M$ is an abelian group then we denote by $\DD(M) := \Spec(\OO[M])$ the diagonalisable group scheme, where the Hopf algebra structure on $\OO[M]$ is given by addition and inversion in $M$.
For any $\OO$-algebra $A$, we have
$$ \DD(M)(A) = \Hom(M, A^{\times}). $$
If $M$ is a finitely generated free abelian group then $\DD(M)$ is a split torus.

\section{Shapiro's lemma} 
\label{sec_shap}

We recall an explicit version of Shapiro's lemma. Let $\Gamma_1$ be a group and let $\Gamma_2$ be a subgroup of $\Gamma_1$ of finite index.
Let $V$ be an abelian group with a $\Gamma_2$-action. We let $\Indu{\Gamma_2}{\Gamma_1}{V}$
be the set of functions $f: \Gamma_1 \rightarrow V$, such that $f(kg)= k f(g)$ for all $k\in \Gamma_2$ and $g\in \Gamma_1$. Then $\Indu{\Gamma_2}{\Gamma_1}{V}$ is naturally an abelian group, isomorphic to a finite 
direct sum of copies of $V$, on which $\Gamma_1$ acts by right translations, that is 
$[g\cdot f](h)\colonequals f(h g)$ for all $g, h\in \Gamma_1$. In this section topology does not play a role, 
all cohomology is just abstract group cohomology, there is no continuity condition on the cocycles.

\begin{prop}\label{shap} Let $\Phi: \Gamma_1 \rightarrow \Indu{\Gamma_2}{\Gamma_1}{V}$ be a 
 $1$-cocycle, then the function $\varphi: \Gamma_2\rightarrow V$, given by 
$$ \varphi(k)=[\Phi(k)](1)$$
is a $1$-cocycle. Moreover, the above map induces an isomorphism in
cohomology
$$H^1(\Gamma_1, \Indu{\Gamma_2}{\Gamma_1}{V})\cong H^1(\Gamma_2, V).$$
\end{prop}
\begin{proof} \cite[Theorem 3.7]{lang}.
\end{proof} 

We will construct an explicit inverse of the map above following
\cite{lang}. For each right coset $c$ of $\Gamma_2$ in $\Gamma_1$ we
fix a coset representative $\overline{c}$, so that the representative
of the trivial coset is $1$. In particular,
$$\Gamma_1= \bigcup_c \Gamma_2 \overline{c}= \bigcup_c \overline{c}^{-1}\Gamma_2.$$ 
Since $c g= \Gamma_2 \overline{cg}$ for every $g\in \Gamma_1$,  we have
$$ \overline{c} g \overline{cg}^{-1}\in \Gamma_2, \quad \forall g\in \Gamma_1.$$
Let $\FF_V$ be the subgroup of $\Indu{\Gamma_2}{\Gamma_1}{V}$ of functions with support in $\Gamma_2$.
By evaluating at $1$ we obtain a canonical isomorphism of
$\Gamma_2$-representations 
between  $\FF_V$ and $V$; the inverse homomorphism is obtained by mapping 
$v$ to the function $g \mapsto g v$.  If $\varphi\in Z^1(\Gamma_2, V)$ then 
 the isomorphism gives us a cocycle $f_{\varphi}\in Z^1(\Gamma_2, \mathcal F_V)$, where
if $h\in \Gamma_2$ then  $f_{\varphi}(h)\in \mathcal \FF_V$ is the function given by 
\begin{equation}\label{f_varphi}
[f_{\varphi}(h)](k)= k \varphi(h), \quad \forall k \in \Gamma_2.
\end{equation}

\begin{lem}\label{inv} Let $\varphi:\Gamma_2\rightarrow V$ be a $1$-cocycle, and let
  $f_{\varphi}: \Gamma_2\rightarrow \FF_V$ be the $1$-cocycle corresponding to
  $\varphi$, via the canonical isomorphism $\FF_V\cong V$ and let
  ${\Phi_{\varphi}:\Gamma_1\rightarrow \Indu{\Gamma_2}{\Gamma_1}{V}}$ be the function given by
$$\Phi_{\varphi}(g)=\sum_c \overline{c}^{-1} f_{\varphi}(\overline{c} g
\overline{cg}^{-1}).$$
Then $\Phi_{\varphi}$ is a $1$-cocycle, such that
$$[\Phi_{\varphi}(k)](1)=\varphi(k), \quad \forall k\in \Gamma_2.$$
\end{lem}

\begin{proof} We have to show that 
$$\Phi_{\varphi}(gh)=\Phi_{\varphi}(g)+ g\Phi_{\varphi}(h),\quad   \forall g, h\in \Gamma_1.$$
It is enough to show that the equality holds once we evaluate both
sides at $k \overline{c}$, where $\overline{c}$ is a representative of
the coset $c$ and $k\in \Gamma_2$. If $f\in \FF_V$ then $\supp
\overline{c}^{-1} f\subseteq c$, hence
$$[\Phi_{\varphi}(gh)](k\overline{c})= [f_{\varphi}(\overline{c} gh
\overline{cgh}^{-1})](k), \quad  [\Phi_{\varphi}(g)](k\overline{c})= [f_{\varphi}(\overline{c} g
\overline{cg}^{-1})](k),$$
\begin{equation}
\begin{split}\notag
[g \Phi_{\varphi}(h)](k\overline{c})=[\Phi_{\varphi}(h)](k\overline{c}g)=&
[f_{\varphi}(\overline{cg} h \overline{cgh}^{-1})](k \overline{c} g
\overline{cg}^{-1})\\
=&[\overline{c}g \overline{cg}^{-1}
f_{\varphi}(\overline{cg} h \overline{cgh}^{-1})](k).
\end{split}
\end{equation} 
Since $f_{\varphi}$ is a $1$-cocycle we have 
$$\overline{c}g \overline{cg}^{-1}
f_{\varphi}(\overline{cg} h \overline{cgh}^{-1})+f_{\varphi}(\overline{c} g
\overline{cg}^{-1})=f_{\varphi}(\overline{c}g
\overline{cg}^{-1}\overline{cg} h
\overline{cgh}^{-1})=f_{\varphi}(\overline{c}g h \overline{cgh}^{-1})$$
and hence $\Phi_{\varphi}$ is a $1$-cocycle. Since the representative of the
trivial coset was chosen to be $1$, we have
$$[\Phi_{\varphi}(k)](1)= [f_{\varphi}(k)](1)=\varphi(k), \quad \forall k\in \Gamma_2.$$
\end{proof}

\begin{remar} The map $\varphi\mapsto \Phi_{\varphi}$ depends on the choice of the coset representatives $\bar{c}$ and so it is not canonical. However once these representatives have been
fixed it is immediate from the formulas that the map is functorial in $V$. If 
$\alpha:V\rightarrow W$ is a $\Gamma_2$-equivariant homomorphism of abelian groups equipped with  $\Gamma_2$-action and $\psi\in Z^1(\Gamma_2, W)$ is the image of $\varphi$ under 
the map $Z^1(\Gamma_2, V)\rightarrow Z^1(\Gamma_2, W)$ induced by $\alpha$ 
 then $\Phi_{\varphi}$ maps to $\Phi_{\psi}\in Z^1(\Gamma_1, 
\Indu{\Gamma_2}{\Gamma_1}{W})$.
\end{remar}

\section{\texorpdfstring{The space of $1$-cocycles}{The space of 1-cocycles}}\label{sec_Z1}

We shall later identify our space of deformations with a space of $1$-cocycles.
In this section we study the space of $1$-cocycles in a more abstract situation.
We fix the following notation:
\begin{itemize}
    \item $\Gamma_1$ is an abstract group.
    \item $M$ is an abelian group equipped with a linear left $\Gamma_1$-action with kernel $\Gamma_2$ of finite index in $\Gamma_1$.
    \item $\Delta \colonequals \Gamma_1/\Gamma_2$. We denote the projection $\pi : \Gamma_1 \to \Delta$.
\end{itemize}
We start by constructing an abelian group $\EE$, which represents the functor
$$ \Ab \to \Ab, ~V \mapsto Z^1(\Gamma_1, \Ind^{\Gamma_1}_{\Gamma_2} V). $$
It turns out, that $\EE$ carries an action of $\Delta$, such that $(\EE \otimes M)_{\Delta}$ represents the functor
$$ \Ab \rightarrow \Ab, ~V \mapsto Z^1(\Gamma_1, \Hom(M, V)). $$
Our strategy is to realize $\Hom(M, V)$ as $\Delta$-invariants 
in $\Ind^{\Gamma_1}_{\Gamma_2}\Hom(M, V)$, and then use the explicit version of Shapiro's lemma of \Cref{sec_shap} to compute $$Z^1(\Gamma_1, \Ind^{\Gamma_1}_{\Gamma_2} \Hom(M, V)).$$
At the end of this section we study the functors $\CRing\rightarrow \Ab$ defined by 
\begin{align*}
     A\mapsto Z^1(\Gamma_1, \Hom(M, A^{\times})), \quad 
    A \mapsto H^1(\Gamma_1, \Hom(M, A^{\times})).
\end{align*}
 
Let $V$ be an abelian group equipped with a trivial $\Gamma_2$-action. We may identify $\Ind^{\Gamma_1}_{\Gamma_2} V$ with the space of functions
 $f: \Delta \rightarrow V$ where the action is given by $[g\cdot f](c)= f(cg)$ for all $g\in \Gamma_1$ and $c\in \Delta$.  
We also have an action of $\Delta$ on $\Ind_{\Gamma_2}^{\Gamma_1} V$, which commutes 
with the action of $\Gamma_1$ and is given by $[d\cdot f](c)= f(d^{-1} c)$ for all $d\in \Delta$. 
This induces an action of $\Delta$ on $Z^1(\Gamma_1, \Ind^{\Gamma_1}_{\Gamma_2} V)$, which is given 
explicitly by 
\begin{equation}\label{action_Delta}
[[d\ast \Phi](g)](c)= [\Phi(g)]( d^{-1} c). 
\end{equation}
Since the action of $\Gamma_2$ on $V$ is trivial we have 
\begin{equation}\label{trivial_action}
Z^1(\Gamma_2, V)= H^1(\Gamma_2, V)= \Hom(\Gamma_2, V),
\end{equation}
where $\Hom$ stands for group homomorphisms. For each $c\in \Delta$ we fix a coset representative 
$\overline{c}\in \Gamma_1$ as in the previous section.

\begin{cor}\label{4_1} Let $\varphi\in \Hom(\Gamma_2, V)$. Then the $1$-cocycle $\Phi_{\varphi}$ constructed in 
\Cref{inv} is given by
$$\Phi_{\varphi}(g)=\sum_{c\in \Delta} \varphi(\overline{c} g \overline{cg}^{-1})\Eins_c,$$
where $\Eins_c$ is the indicator function for the coset $c$.
\end{cor}

\begin{proof} Since the action of $\Gamma_2$ on $V$ is trivial by assumption it follows from 
\eqref{f_varphi} that $[f_{\varphi}(h)](k)= \varphi(h)$ for all $k, h\in \Gamma_2$. 
The assertion then follows from \Cref{inv}.
\end{proof}

\Cref{shap} and \eqref{trivial_action} give us an exact sequence of abelian groups: 
$$0\rightarrow B^1(\Gamma_1, \Ind^{\Gamma_1}_{\Gamma_2} V) \rightarrow 
Z^1(\Gamma_1, \Ind^{\Gamma_1}_{\Gamma_2} V)\rightarrow \Hom(\Gamma_2, V)\rightarrow 0$$
and \Cref{inv} gives us a section so that 
\begin{equation}\label{section_shap}
Z^1(\Gamma_1, \Ind^{\Gamma_1}_{\Gamma_2} V)\cong B^1(\Gamma_1, \Ind^{\Gamma_1}_{\Gamma_2} V)\oplus 
 \Hom(\Gamma_2, V)
 \end{equation}
We want to understand the action of $\Delta$ on  $B^1(\Gamma_1, \Ind^{\Gamma_1}_{\Gamma_2} V)\oplus \Hom(\Gamma_2, V)$ via this isomorphism. It follows from \eqref{action_Delta} that 
$[[d\ast \Phi](g)](dc) = [\Phi(g)](c)$. Hence 
$$[d\ast \Phi_{\varphi}](g)=\sum_{c\in \Delta} \varphi(\overline{c} g \overline{cg}^{-1})\Eins_{dc}.$$
for $\varphi\in \Hom(\Gamma_2, V)$.
Since $\Gamma_2$ is normal in $\Gamma_1$, $c k = c$ for all $k\in \Gamma_2$  and all $c\in \Delta$. 
Hence, $\overline{ck}= \overline{c}$ and we conclude that 
$$[[d\ast \Phi_{\varphi}](k)](1)=\sum_c \varphi( \overline{c} k \overline{c}^{-1}) \Eins_{dc}(1)= 
\varphi( \overline{d^{-1}} k \overline{d^{-1}}^{-1})=\varphi(\overline{d}^{-1} k \overline{d})=:[d\cdot \varphi](k).$$
We note that  $\varphi((\overline{d} h)^{-1} k (\overline{d} h))= \varphi(\overline{d}^{-1} k \overline{d})$ for all $h,k\in \Gamma_2$, which justifies the third equality and we take the last equality as the definition $d\cdot \varphi$.
Since $[[\Phi_{d\cdot\varphi}](k)](1)=[d\cdot \varphi](k)$ by \Cref{inv}, \Cref{shap} implies that
the cocycle  $d\ast \Phi_{\varphi} - \Phi_{d\cdot \varphi}$ 
is a $1$-coboundary and thus there is a function $f:\Delta\rightarrow V$ depending on $d$ such that for all $g\in \Gamma_1$ we have 
\begin{equation}\label{define_f}
 [d\ast \Phi_{\varphi}](g) - \Phi_{d\cdot \varphi}(g)= (g-1) f.
\end{equation}
After subtracting a constant function we may assume that $f(1)=0$. By evaluating 
both sides of \eqref{define_f} at $1$, we obtain: 
\begin{equation}\label{find_f}
\begin{split}
f(g) &= f(g)-f(1)= [[d\ast \Phi_{\varphi}](g)](1) - [\Phi_{d\cdot \varphi}(g)](1)\\
&= \varphi( \overline{d^{-1}} g \overline{d^{-1}g}^{-1}) - [d\cdot\varphi](g \overline{g}^{-1})\\
&=\varphi( \overline{d^{-1}} g \overline{d^{-1}g}^{-1}) - \varphi( \overline{d^{-1}} g \overline{g}^{-1} \overline{d^{-1}}^{-1})\\
&=\varphi( (\overline{d^{-1}} g \overline{g}^{-1} \overline{d^{-1}}^{-1})^{-1} \overline{d^{-1}} g \overline{d^{-1}g}^{-1})\\
&=\varphi( \overline{d^{-1}} \overline{g} \overline{d^{-1} g}^{-1}).
\end{split}
\end{equation}
We note that in the above calculation $\overline{g}$ is the fixed coset representative of the 
coset $\Gamma_2 g$. In particular, $\overline{g}\in \Gamma_1$ and not in $\Delta$. The function 
$f$ satisfies $f(kg)=f(g)$ for all $k\in \Gamma_2$, so we may consider as a function on $\Delta$.
In this case the formula \eqref{find_f} applied with $g=\overline{c}$ gives us 
\begin{equation}\label{find_f_c}
f(c)= \varphi( \overline{d^{-1}} \overline{c} \overline{d^{-1} c}^{-1}), \quad \forall c\in \Delta.
\end{equation}
The function $f$ corresponds to a group homomorphism $\alpha_f: \ZZ[\Delta]\rightarrow V$, 
given by $\alpha_f(c)= f(c)$. Since $f(1)=0$ the restriction of $\alpha_f$ to the augmentation ideal 
$I_{\Delta}$ is given by 
\begin{equation}\label{define_alpha_f}
\alpha_f( c-1) = \varphi( \overline{d^{-1}} \overline{c} \overline{d^{-1} c}^{-1}), \quad \forall c\in \Delta.
\end{equation}

Since $V$ is abelian we have $\Hom(\Gamma_2, V)= \Hom(\Gamma_2^{\ab}, V)$, where $\Gamma_2^{\ab}$ is the maximal 
abelian  quotient of $\Gamma_2$. For $d, c\in \Delta$ we define 
$\kappa(d, c)$ to be the image of $\overline{d} \overline{c} \overline{dc}^{-1}$ in $\Gamma_2^{\ab}$.
The map $\kappa: \Delta\times \Delta \rightarrow \Gamma_2^{\ab}$ 
is the $2$-cocycle corresponding to the extension 
\begin{equation*}
 0\rightarrow  \Gamma_2^{\ab}\rightarrow \Gamma_1/[\Gamma_2, \Gamma_2] \rightarrow \Delta\rightarrow 0.
\end{equation*} 
We define an action of $\Delta$ on $\Gamma_2^{\ab}\oplus I_{\Delta}$ by letting 
\begin{equation}\label{ast_action}
    \begin{split}
        d\ast ( g, c-1): = ( d g d^{-1} \cdot \kappa(d, c), d c - d).
    \end{split}
\end{equation} 
This action defines an action of $\Delta$ on $\Hom(\Gamma_2^{\ab}\oplus I_{\Delta}, V)$ by 
\begin{equation}\label{ast_dual}
\begin{split}
[ d\ast (\varphi, \alpha)]( g, c-1)&= (\varphi, \alpha) ( d^{-1}\ast (g, c-1))\\
&= \varphi(d^{-1} g d) + \varphi(\kappa(d^{-1}, c))+ \alpha( d^{-1}c -d^{-1}).
\end{split}
\end{equation}

We obtain an exact sequence of $\ZZ[\Delta]$-modules 
\begin{equation}\label{Eextension}
    \begin{split}
        0\rightarrow \Gamma_2^{\ab} \rightarrow \mathcal E \rightarrow I_{\Delta}\rightarrow 0.
    \end{split}
\end{equation}
by letting $\mathcal E= \Gamma_2^{\ab}\oplus I_{\Delta}$ with the $\Delta$-action defined as above.

\begin{remar}\label{CE_rmk} Using the exact sequence 
    $0\rightarrow I_{\Delta}\rightarrow \ZZ[\Delta]\rightarrow \ZZ\rightarrow 0$ 
    of $\ZZ[\Delta]$-modules one obtains a canonical identification 
    $$ \Ext^1_{\ZZ[\Delta]}(I_{\Delta}, \Gamma_2^{\ab}) \cong \Ext^2_{\ZZ[\Delta]}(\ZZ, \Gamma_2^{\ab})\cong H^2(\Delta, \Gamma_2^{\ab}), $$
    see \cite[Chapter XIV, §4, Remark]{CE}.
    The extension class of \eqref{Eextension} gives an element in 
    $\Ext^1_{\ZZ[\Delta]}( I_{\Delta},\Gamma_2^{\ab})$. One can show, that the 
    image of this class in $H^2(\Delta, \Gamma_2^{\ab})$ is equal to the class of $\kappa$.
\end{remar}

To $\alpha \in \Hom(I_{\Delta}, V)$ we may associate a function $f_{\alpha}: \Delta\rightarrow V$ by letting $f_{\alpha}(c)\colonequals \alpha(c-1)$. We note that $f_{\alpha}(1)=0$. We then define 
a $1$-coboundary $b_{\alpha}\in B^1(\Gamma_1, \Ind^{\Gamma_1}_{\Gamma_2} V)$ 
by $b_{\alpha}(g)\colonequals (g-1) f_{\alpha}$. We may recover $f_{\alpha}$ from $b_{\alpha}$ as the 
unique function $f:\Delta\rightarrow V$ satisfying $f(1)=0$ and $b_{\alpha}(g)= (g-1)f$ for 
all $g\in \Gamma_1$. Thus the map 
\begin{equation}\label{B_one}
 \Hom(I_{\Delta}, V) \rightarrow B^1(\Gamma_1, \Ind^{\Gamma_1}_{\Gamma_2} V), \quad \alpha \mapsto b_{\alpha}
\end{equation} 
is an isomorphism. We now record the consequence of our calculations: 

\begin{prop}\label{univ_prop_of_EE}
    If $V$ is an abelian group  with the trivial $\Gamma_2$-action then
    sending $(\varphi, \alpha)$ to $\Phi_{\varphi}+b_{\alpha}$ induces an isomorphism 
    $$ \Hom ( \Gamma_2^{\ab}\oplus I_{\Delta}, V) \overset{\cong}{\longrightarrow}  
    Z^1(\Gamma_1, \Ind^{\Gamma_1}_{\Gamma_2} V),$$
    which is $\Delta$-equivariant for the actions defined in \eqref{ast_dual} and \eqref{action_Delta}.
\end{prop}

\begin{proof}
    It follows from \eqref{B_one} and \eqref{section_shap} that the map is an 
    isomorphism of abelian groups. We have to check that it is $\Delta$-equivariant.
    We have 
    \begin{equation*}
    \begin{split}
    [[d\ast b_{\alpha}](g)](c)&= [b_{\alpha}(g)](d^{-1} c)= f_{\alpha}(d^{-1}c g)-
    f_{\alpha}(d^{-1}c)\\
    &= \alpha( d^{-1}c g -1) - \alpha( d^{-1}c -1)=\alpha( d^{-1}c g - d^{-1} c),
    \end{split}
    \end{equation*}
    \begin{equation*}
    \begin{split}
    [b_{d\ast \alpha}(g)](c)&=f_{d\ast \alpha}(cg) - f_{d\ast\alpha}(c)=
    [d\ast \alpha](cg-1) - [d\ast \alpha]( c-1)\\
    &= \alpha(d^{-1} cg -d^{-1}) - \alpha(d^{-1}c -d^{-1})=\alpha(d^{-1} cg -d^{-1}c).
    \end{split}
    \end{equation*}
    Hence, $d\ast b_{\alpha} = b_{d\ast \alpha}$. It follows from \eqref{define_f}, \eqref{find_f_c} and 
    \eqref{define_alpha_f} that 
    $d\ast \Phi_{\varphi} = \Phi_{d\cdot \varphi} + b_{\beta}$, where $\beta( c-1)= \varphi(\kappa(d^{-1}, c))$ for 
    all $c\in \Delta$. On the other hand, it follows from \eqref{ast_dual} that 
    $d\ast (\varphi, 0) = (d\cdot \varphi, \beta)$. Hence, 
    $d\ast (\varphi, 0)$ is mapped to $\Phi_{d\cdot\varphi} + b_{\beta} = d\ast \Phi_{\varphi}$. 
\end{proof}

Recall, that $M$ is an abelian group with an action of $\Delta$, so $\Gamma_2$ acts trivially on $\Hom(M, V)$.
We have natural isomorphisms
of $\Gamma_1$-representations
\begin{equation}\label{induction}
\Ind^{\Gamma_1}_{\Gamma_2} \Hom( M, V) \cong \Hom ( \ZZ[\Delta], \Hom(M, V)) \cong \Hom(\ZZ[\Delta]\otimes M, V),
\end{equation}
where the action of $\Gamma_1$ on the last term is induced by its action on $\ZZ[\Delta]\otimes M$, which 
is given by $g \cdot (c \otimes m) = cg^{-1} \otimes m$, so that $[g \cdot f](x)=f(g^{-1}x)$ for all $x\in \ZZ[\Delta]\otimes M$. 
The diagonal action of $\Delta$ on $\ZZ[\Delta]\otimes M$ is given by $d\cdot (c\otimes m) = dc \otimes dm$, commutes with the action of $\Gamma_1$ and 
induces a (left) action of $\Delta$ on $\Ind^{\Gamma_1}_{\Gamma_2} \Hom( M, V)$, which commutes with the action of 
$\Gamma_1$. If $M=\ZZ$ with the trivial $\Delta$-action then the construction recovers the action of 
$\Delta$ on $\Ind^{\Gamma_1}_{\Gamma_2} V$ considered earlier. 

\begin{lem}\label{proj_form} As $\Gamma_1$-representations $(\Ind^{\Gamma_1}_{\Gamma_2} \Hom( M, V))^{\Delta}$  
is naturally isomorphic to $\Hom(M, V)$ with $\Gamma_1$ acting on it via $\Delta$.
\end{lem}

\begin{proof} The map $\ZZ[\Delta]\otimes M \rightarrow \ZZ[\Delta]\otimes M$ , $c\otimes m \mapsto c\otimes c^{-1}m$ is an isomorphism of $\Delta \times \Gamma_1$-representations, where on the source the action 
is the one considered above, and on the target $\Delta$ acts by $d \ast ( c\otimes m) = dc \otimes m$ and 
$\Gamma_1$ acts by $g \ast (c\otimes m) = c g^{-1} \otimes g m$. 
It is immediate by considering the $\ast$-action that $(\ZZ[\Delta]\otimes M)_{\Delta} \cong M$ and
$\Gamma_1$ acts on $M$  via $\Delta$. Using \eqref{induction} we can then translate this statement 
to a statement about invariants. 
\end{proof}

\begin{prop}\label{Z1withE} Let $\mathcal E$ be the $\ZZ[\Delta]$-module constructed above. Then for all abelian groups $V$, we 
have an isomorphism of abelian groups functorial in $V$: 
\begin{align}
    \Hom( (\mathcal E\otimes M)_{\Delta}, V) \cong Z^1(\Gamma_1, \Hom(M, V)), \label{Z1withE_iso}
\end{align}
where the action of $\Delta$ on $\mathcal E\otimes M$ is diagonal and $\Gamma_1$ acts on $\Hom(M, V)$
via $\Delta$.
\end{prop}

\begin{proof} \Cref{univ_prop_of_EE} gives us  $\Delta$-equivariant isomorphisms: 
\begin{equation*}
\Hom( \mathcal E\otimes M, V) \cong \Hom ( \mathcal E, \Hom(M, V)) \cong Z^1(\Gamma_1, \Ind^{\Gamma_1}_{\Gamma_2} \Hom(M, V)).
\end{equation*}
The assertion follows after taking $\Delta$-invariants and applying \Cref{proj_form}.
\end{proof}

The exact sequence $0\rightarrow I_{\Delta}\rightarrow \ZZ[\Delta]\rightarrow \ZZ\rightarrow 0$ 
remains exact after tensoring with $M$. Taking $\Delta$-coinvariants yields an exact sequence 
\begin{equation*}
0\rightarrow H_1(\Delta, M)\rightarrow (I_{\Delta}\otimes M)_{\Delta} \rightarrow M \rightarrow M_{\Delta}\rightarrow 0.
\end{equation*}

This gives us a surjection $(I_{\Delta}\otimes M)_{\Delta} \twoheadrightarrow I_{\Delta} M$. 
By composing this map with the surjection $(\mathcal E \otimes M)_{\Delta} \twoheadrightarrow
(I_{\Delta}\otimes M)_{\Delta}$ induced by 
\eqref{Eextension} we obtain a surjection $(\mathcal E\otimes M)_{\Delta}\twoheadrightarrow I_{\Delta} M$. 
This in turn yields an injection 
$$\Hom(I_{\Delta} M, V)\hookrightarrow \Hom((\mathcal E\otimes M)_{\Delta}, V).$$ 

\begin{lem}\label{image_boundaries} The isomorphism \eqref{Z1withE_iso} in \Cref{Z1withE} identifies the space
of $1$-co\-boun\-daries $B^1(\Gamma_1, \Hom(M, V))$ with  a subgroup of $\Hom(I_{\Delta} M, V)$.
Moreover, it induces an isomorphism between the two groups if $V$ is divisible. 
\end{lem} 

\begin{proof}
    The isomorphism $(\ZZ[\Delta]\otimes M)_{\Delta}\cong M$ for the diagonal action of $\Delta$
    on $\ZZ[\Delta]\otimes M$ is realised by the map $c\otimes m \mapsto c^{-1} m$. The image 
    of $I_{\Delta}\otimes M$ under this map is equal to $I_{\Delta}M$. Thus 
    the isomorphism $$\vartheta: \Hom(M, V) \overset{\cong}{\longrightarrow} (\Ind^{\Gamma_1}_{\Gamma_2} \Hom(M,V))^{\Delta}$$ in \Cref{proj_form} is given explicitly by 
    $$[\vartheta(\alpha)(c)](m)=\alpha(c^{-1}m), \quad \forall c\in \Delta, \quad \forall m\in M .$$
    
    Let $b\in B^1(\Gamma_1, \Hom(M,V))$ be a boundary with $b(g)=(g-1)\alpha$ for all $g\in \Gamma_1$. Its image in 
    $B^1(\Gamma_1, \Ind^{\Gamma_1}_{\Gamma_2}\Hom(M,V))$ is the boundary 
    $b'$ given by $b'(g)= (g-1) \vartheta(\alpha)$. The constant function $f'(c)=\alpha$ for all $c\in \Delta$
    is $\Gamma_1$-invariant thus $$b'(g)= (g-1)( \vartheta(\alpha) - f'), \quad \forall g\in \Gamma_1.$$ 
    Since $(\vartheta(\alpha)-f')(1)= \alpha -\alpha =0$, by \Cref{univ_prop_of_EE} $b'=b_{\beta}$, where
    $\beta: I_{\Delta}\rightarrow \Hom(M, V)$ is given by 
    $$ [\beta(c-1)](m)= [(\vartheta(\alpha)- f')(c)](m)= \alpha(c^{-1}m)-\alpha(m)= \alpha(c^{-1}m -m).$$
    We conclude that the image of $B^1(\Gamma_1, \Hom(M,V))$ under the 
    isomorphism in \Cref{Z1withE} is contained in $\Hom(I_{\Delta}M, V)$.  
    
    Conversely, if we start with a homomorphism $\beta': I_{\Delta}M \rightarrow V$ then 
    by pulling it back under the surjection $I_{\Delta}\otimes M \twoheadrightarrow I_{\Delta}M$ 
    we obtain a homomorphism $\beta: I_{\Delta}\otimes M \rightarrow V$ given by 
    $\beta( (c-1)\otimes m)= \beta'(c^{-1}m -m)$. The corresponding coboundary 
    $b_{\beta} \in B^1(\Gamma_1, \Ind^{\Gamma_1}_{\Gamma_2}\Hom(M,V))$ is given by 
    $b_{\beta}(g)= (g-1)f_{\beta}$, where
    $$[f_{\beta}(c)](m)=\beta((c-1)\otimes m)= \beta'(c^{-1} m -m).$$  
    If $V$ is divisible then $V$ is injective in $\Ab$ and the exact sequence $$0\rightarrow I_{\Delta}M \rightarrow M \rightarrow M_{\Delta}\rightarrow 0$$ gives rise to an exact sequence 
    $$0\rightarrow \Hom(M_{\Delta}, V)\rightarrow \Hom(M, V)\rightarrow \Hom (I_{\Delta}M, V)\rightarrow 0$$
    and we conclude that there is $\alpha\in \Hom(M, V)$ mapping to $\beta'$, so that 
    $$ \beta'(c^{-1} m -m)=\alpha(c^{-1}m) - \alpha(m).$$
    Let $f'$ be the constant function $f'(c)=\alpha$ for all $c\in \Delta$. Then 
    the boundary $g\mapsto (g-1)(f_{\beta}+f')$ is equal to $b_{\beta}$ and
    $f_{\beta}+f'= \vartheta(\alpha)$.
\end{proof}

By taking $V= (\mathcal E\otimes M)_{\Delta}$ in \Cref{Z1withE} we obtain a obtain a natural cocycle 
$$\Phi_{\nat}\in Z^1(\Gamma_1, \Hom(M, (\mathcal E\otimes M)_{\Delta})),$$ 
which corresponds to the identity in $\Hom( (\mathcal E\otimes M)_{\Delta}, (\mathcal E\otimes M)_{\Delta})$. 

\begin{lem}\label{lem:univ_coeff}
    If $V$ is any abelian group then for all $\Gamma_1$-modules $N$ and all $i\ge 0$ we have a canonical map
    \begin{equation}\label{univ_coeff}
    H^i(\Gamma_1, \Hom(N, V)) \rightarrow \Hom( H_i(\Gamma_1, N), V),
    \end{equation}
    which is an isomorphism if $V$ is divisible. 
\end{lem}

\begin{proof} 
    If $(C_{\bullet}, d_{\bullet})$ is a complex of abelian groups then we
    let $Z_n = \ker(d_n:C_n\rightarrow C_{n-1})$ and $Z^n=\ker(d_{n+1}^{\ast}: \Hom(C_n, V)\rightarrow 
    \Hom(C_{n+1}, V))$. The evaluation pairing $Z_n \times Z^n\rightarrow V$ induces a bilinear map 
    $H_n(C_{\bullet})\times H^n(\Hom(C_{\bullet}, V))\rightarrow V$, which induces 
    a canonical map
    $H^n(\Hom(C_{\bullet}, V))\rightarrow \Hom (H_n(C_{\bullet}), V)$. If $V$ is divisible then $V$ is 
    injective in $\Ab$ and the map is an isomorphism.
    
    Let $P_{\bullet}\twoheadrightarrow \ZZ$ be a resolution of $\ZZ$ by projective $\ZZ[\Gamma_1]$-modules. The complex $\Hom_{\Gamma_1}(P_{\bullet}, \Hom(N, V))\cong 
    \Hom ( (P_{\bullet}\otimes N)_{\Gamma_1}, V)$ computes the cohomology groups $H^i(\Gamma_1, \Hom(N, V))$. 
    The complex $(P_{\bullet}\otimes N)_{\Gamma_1}$ computes the homology groups $H_i(\Gamma_1, N)$.
    We apply the previous discussion to $C_{\bullet}\colonequals (P_{\bullet}\otimes N)_{\Gamma_1}$ to obtain the 
    required homomorphism.
\end{proof}

Applying \Cref{lem:univ_coeff} with $N=M$ and $V=(\mathcal E\otimes M)_{\Delta}$
we obtain a homomorphism $\varphi_{\nat}: H_1(\Gamma_1, M)\rightarrow (\mathcal E \otimes M)_{\Delta}$ corresponding 
to the cohomology class $[\Phi_{\nat}]$.

\begin{lem}\label{image_nat} For all abelian groups $V$ the composition
\begin{equation}\label{eq_nat}
\begin{split}
\Hom((\mathcal E \otimes M)_{\Delta}, V) & \overset{\eqref{Z1withE_iso}}{\longrightarrow} Z^1(\Gamma_1, \Hom(M, V))\twoheadrightarrow \\
&H^1(\Gamma_1, \Hom(M, V)) \xrightarrow{\eqref{univ_coeff}} \Hom(H_1(\Gamma_1, M), V)
\end{split}
\end{equation}
is given by $\psi\mapsto \psi\circ \varphi_{\nat}$.
\end{lem}

\begin{proof} 
    This follows from the fact that all our constructions are functorial in $V$. If 
    we denote the four functors appearing in \eqref{eq_nat} with $A$, $B$, $C$, $D$ then 
    for a homomorphism of abelian groups $\psi:W\rightarrow V$ we obtain a diagram:
    \begin{center}
        \begin{tikzcd}
            A(W)\arrow[d, "\psi\circ"] \arrow[ r] & 
            B(W) \arrow[d, "\psi\circ"] \arrow[r] & C(W) \arrow[d, "\psi\circ"] \arrow[r] & D(W)
            \arrow[d, "\psi\circ"] \\
            A(V)\arrow[r] & B(V)\ar[r] & C(V)\ar[r]& D(V)
        \end{tikzcd}
    \end{center}
    with commutative squares. If we take $W=(\mathcal E \otimes M)_{\Delta}$ then 
    the identity map in $A((\mathcal E \otimes M)_{\Delta})$ maps to $\varphi_{\nat}$ in $D((\mathcal E \otimes M)_{\Delta})$ by construction and hence to $\psi\circ \varphi_{\nat}$ in $D(V)$. Since
    the identity maps to $\psi$ in $A(V)$ we obtain the assertion.
\end{proof}

\begin{prop}\label{prop_represent} There is an exact sequence of abelian groups
\begin{equation}\label{exact_seq_nat}
0\rightarrow H_1(\Gamma_1, M) \xrightarrow{\varphi_{\nat}} (\mathcal E \otimes M)_{\Delta}
\xrightarrow{q} I_{\Delta} M\rightarrow 0
\end{equation}
which is functorial in $M$ and for all abelian groups $V$ the diagram 
\begin{center}
    \begin{tikzcd}
        B^1(\Gamma_1, \Hom(M, V)) \arrow[hookrightarrow, d]\arrow[hookrightarrow, r] & Z^1(\Gamma_1, \Hom(M, V))\ar[r,"\eqref{univ_coeff}"]\arrow[d, "\cong"]  & \Hom( H_1(\Gamma_1, M), V) \arrow[d,"="]\\
        \Hom (I_{\Delta}M, V)  \arrow[hookrightarrow, r] & 
        \Hom((\mathcal E\otimes M)_{\Delta}, V) \arrow[r, "\varphi_{\nat}^{\ast}"] & \Hom(H_1(\Gamma_1, M), V)   
        \end{tikzcd}
    \end{center}
commutes. 
\end{prop} 

\begin{proof} 
    Lemmas \ref{image_boundaries} and \ref{image_nat} give us the commutative diagram 
    above. Moreover, if $V$ is divisible then by \Cref{lem:univ_coeff} the top row is exact, the last arrow in the top row is surjective
    and the first vertical arrow is an isomorphism. 
    We deduce that for all divisible $V$ the maps $q:(\mathcal E\otimes M)_{\Delta}\twoheadrightarrow I_{\Delta}M$ and $\varphi_{\nat}: H_1(\Gamma, M)\rightarrow (\mathcal E\otimes M)_{\Delta}$ induce an exact sequence
    \begin{equation}\label{eq_hom_V}
    0\rightarrow \Hom(I_{\Delta}M, V)\rightarrow \Hom((\mathcal E\otimes M)_{\Delta}, V)\rightarrow 
    \Hom(H_1(\Gamma, M), V)\rightarrow 0.
    \end{equation}
    Since 
    divisible groups are precisely injective objects in $\Ab$, which has enough injectives, an 
    abelian group $A$ is zero if and only if $\Hom(A, V)=0$ for all divisible groups $V$. Using
    this and the exactness of \eqref{eq_hom_V} we obtain that $\mathrm{Im}(q\circ \varphi_{\nat})=0$, thus 
    \eqref{exact_seq_nat} is a complex and a further application of the same argument shows that 
    \eqref{exact_seq_nat} is exact.
\end{proof}

\begin{cor}\label{cor_rings_Z1} The functor $\CRing \rightarrow \Ab$ given by 
$$A\mapsto Z^1(\Gamma_1, \Hom(M, A^{\times}))$$
is represented by the group algebra $\ZZ[(\mathcal E\otimes M)_{\Delta}]$.  
\end{cor}

\begin{proof} 
    If $W$ is an abelian group then we may identify $W$ with a subgroup 
    of units of the group ring $\ZZ[W]$. The map 
    \begin{equation*}\label{rings_groups}
    \Hom_{\CRing}(\ZZ[W], A) \rightarrow \Hom( W, A^{\times}), \quad \psi \mapsto \psi|_W
    \end{equation*}
    is an isomorphism. Applying this observation to 
    $W=(\mathcal E \otimes M)_{\Delta}$  and using \Cref{Z1withE} yields the assertion. 
\end{proof} 

\begin{cor}\label{M_ffr} If $M$ is a free $\ZZ$-module of finite rank then 
\begin{align}
    \ZZ[(\mathcal E\otimes M)_{\Delta}]\cong \ZZ[H_1(\Gamma_1, M)] [t_1^{\pm 1}, \ldots, t_s^{\pm 1}], \label{poly_over_H1}
\end{align}
where $s =\rank_{\ZZ} M -\rank_{\ZZ} M_{\Delta}$.
\end{cor} 

\begin{proof}
    Since $I_{\Delta} M$ is a submodule of $M$, which is free of finite rank over $\ZZ$, 
    we have an isomorphism $I_{\Delta} M \cong \ZZ^{s}$, where $s$ is as above. By choosing 
    a section to the surjection $(\mathcal E\otimes M)_{\Delta}\twoheadrightarrow I_{\Delta}M$ in 
    \eqref{exact_seq_nat} we obtain an isomorphism
    $(\mathcal E\otimes M)_{\Delta} \cong H_1(\Gamma_1, M) \times \ZZ^s$ and this implies the assertion. 
\end{proof}

\begin{prop}\label{H1_sheafify} The map \eqref{univ_coeff} induces a map of presheaves $\CRing \to \Set$
    \begin{align}
        H^1(\Gamma_1, \Hom(M, (-)^{\times})) \to \Spec(\ZZ[H_1(\Gamma_1,M)]) \label{fpqc_sheafification}
    \end{align}
    which exhibits the right hand side as the fpqc sheafification of the left hand side.
\end{prop}

\begin{proof}
    The right hand side is an fpqc sheaf, since it is representable \stackcite{023Q}.
    So it suffices to find for every ring $A$ and every homomorphism ${f : \ZZ[H_1(\Gamma_1,M)] \to A}$ an fpqc cover of $A$ and a descent datum in the left hand side, which maps to a descent datum in the right hand side, which descends to $f$.
    By \cite[Lemma 4.1.1]{Zou} for every ring $A$ there is a faithfully flat map $A \to B$, such that $B^{\times}$ is divisible. By \Cref{lem:univ_coeff} the map
    $$ H^1(\Gamma_1, \Hom(M, B^{\times})) \rightarrow \Hom_{\CRing}(\ZZ[H_1(\Gamma_1, M)], B) $$
    is bijective. So the canonical descent datum associated with $f$ and the map $A \to B$ comes from a descent datum in the right hand side.
\end{proof}

\section{Admissible representations}
\label{sec_adm_rep}

In this section, let $\OO$ be an arbitrary commutative ring, such that 
$\Spec \OO$ is connected. Let $G$ be a generalised torus over $\OO$, such that $G^0$ is split and such that $G/G^0$ is constant.
Let $\Pi : G \to G/G^0$ be the projection map.
We define the character lattice of $G^0$ by 
\begin{align}
    M \colonequals \Hom_{\OO\hyphen\mathrm{GrpSch}}(G^0, \Gm). \label{character_lattice}
\end{align}
We may identify  $G^0$  with the split torus $\DD(M)$ introduced in \Cref{gen_tori}.
We can write $G/G^0 = \underline\Delta$ for a finite group $\Delta$, where $\underline{\Delta} \colonequals \Spec(\Map(\Delta, \OO))$.
For any $\OO$-algebra $A$, we have a natural map $\Delta \to \underline\Delta(A) = \Hom_{\OO\hyphen\alg}(\Map(\Delta,\OO), A)$ defined by evaluation.
The $\OO$-group scheme $G$ acts on $G^0$ by conjugation  and this action factors over $G/G^0 = \underline\Delta$.
So we have a well-defined action map $G/G^0 \times G^0 \to G^0, ~(g,h) \mapsto ghg^{-1}$ with the property that for every $\OO$-algebra $A$ and every $g \in (G/G^0)(A)$, the map $G^0(A) \to G^0(A), ~h \mapsto g h g^{-1}$ is a group automorphism of $G^0(A)$.  This induces a (left) action of $\Delta$ on $M$ via \eqref{character_lattice}.

Let $\Gamma_1$ be an abstract group and let $\pi : \Gamma_1 \twoheadrightarrow \Delta$ be a surjective homomorphism with kernel $\Gamma_2$.
So we are in the situation of \Cref{sec_Z1}. 

\begin{defi}
    We say that a representation $\rho: \Gamma_1\rightarrow G(A)$ is \textit{admissible} if $\Pi \circ \rho : \Gamma_1 \to \underline{\Delta}(A)$, sends
    $\gamma \in \Gamma_1$ to the image of $\pi(\gamma) \in \Delta$ in $\underline\Delta(A)$.
\end{defi}
This terminology is motivated by admissible Galois representations into $L$-groups appearing in the Langlands correspondence, see \cite[\S9]{borel_corvallis}.

Let $\Rep^{\Gamma_1}_{G,\pi}(A)$ be the set of admissible representations $\rho: \Gamma_1\rightarrow G(A)$. 
The group $\DD(M)(A)$ acts on $\Rep^{\Gamma_1}_{G,\pi}(A)$ by conjugation. This defines a scheme-theoretic action
\begin{equation}\label{DM_action_defi}
\DD(M) \times \Rep^{\Gamma_1}_{G,\pi}\rightarrow \Rep^{\Gamma_1}_{G,\pi}.
\end{equation}

\begin{prop}\label{Z1_Rep}
    Assume, there is a representation $\rho_0 \in \Rep^{\Gamma_1}_{G,\pi}(\OO)$.
    Then 
    \begin{align}
        Z^1(\Gamma_1, \DD(M)(A)) \to \Rep^{\Gamma_1}_{G,\pi}(A), \quad \Phi \mapsto [\gamma \mapsto \Phi(\gamma) \rho_0(\gamma)] \label{eq_bij_Z1}
    \end{align}
    is a  bijection, which is natural in $A \in \OO\hyphen\alg$.
    In particular, $\Rep^{\Gamma_1}_{G,\pi}$ is representable by an $\OO$-algebra and \eqref{eq_bij_Z1} induces a natural isomorphism
    \begin{align}
        \OO(\Rep^{\Gamma_1}_{G,\pi}) \eqto \OO[(\mathcal E\otimes M)_{\Delta}]. \label{ORep_iso}
    \end{align}
    Moreover, \eqref{eq_bij_Z1} is $\DD(M)(A)$-equivariant and induces a natural bijection between the set of $\DD(M)(A)$-orbits in $\Rep^{\Gamma_1}_{G,\pi}(A)$ and $H^1(\Gamma_1, \DD(M)(A))$. In particular, \eqref{ORep_iso} is $\DD(M)$-equivariant.
\end{prop}

\begin{proof}
    For $\Phi \in Z^1(\Gamma_1, \DD(M)(A))$, we have
    \begin{equation}\label{Z1_computation}
        \begin{split}
        \Phi(\gamma_1\gamma_2) \rho_0(\gamma_1\gamma_2) &= \Phi(\gamma_1){}^{\gamma_1}\Phi(\gamma_2)\rho_0(\gamma_1)\rho_0(\gamma_2) \\
        &= \Phi(\gamma_1) \rho_0(\gamma_1) \Phi(\gamma_2)\rho_0(\gamma_2), 
        \end{split}
    \end{equation}
    where in the last equality we use that $\rho_0$ is admissible. Thus $\Phi\rho_0$ is a homomorphism. By applying the projection $\Pi : G(A) \to \underline\Delta(A)$, we verify that $\Phi\rho_0 \in \Rep^{\Gamma_1}_{G,\pi}(A)$.
    It is clear, that \eqref{eq_bij_Z1} is a bijection $\Map(\Gamma_1, G(A)) \to \Map(\Gamma_1, G(A))$ with inverse $\rho \mapsto \rho\rho_0^{-1}$.
    If $\rho \in \Rep^{\Gamma_1}_{G,\pi}(A)$, we see that $\Phi \colonequals \rho\rho_0^{-1}$ is a $1$-cocycle by reverting the computation in \eqref{Z1_computation}. The claim about representability follows from \Cref{cor_rings_Z1}.

    The $\DD(M)(A)$-conjugacy classes are the $B^1(\Gamma_1, \DD(M)(A))$-orbits: if we write an admissible homomorphism $\rho : \Gamma_1 \to G(A)$ as above as $\rho(\gamma) = \Phi(\gamma)\rho_0(\gamma)$ and $g \in \DD(M)(A)$, then
    \begin{equation*}
    \begin{split}
        g\rho(\gamma) g^{-1} &= g\Phi(\gamma) ({}^{\gamma}g^{-1}) \rho_0(\gamma) \\
        &= \beta_g(\gamma) \Phi(\gamma) \rho_0(\gamma),
    \end{split}
    \end{equation*}
    where $\beta_g \in B^1(\Gamma_1, \DD(M)(A))$ is the coboundary $\beta_g(\gamma)= g ({}^{\gamma}g^{-1})$. It follows that the $\DD(M)(A)$-orbits are the classes in $H^1(\Gamma_1, \DD(M)(A))$.
\end{proof}

\begin{remar}
    If $G = G^0 \rtimes \underline\Delta$, then the composition of $\pi$ with the natural map $\Delta \to \underline\Delta(\OO) \to G^0(\OO) \rtimes \underline\Delta(\OO) = G(\OO)$ is a canonical choice for a representation $\rho_0 \in \Rep^{\Gamma_1}_{G,\pi}(\OO)$. In general, the isomorphisms \eqref{eq_bij_Z1} and \eqref{ORep_iso} depend on $\rho_0$.
\end{remar}

\begin{prop}\label{action_corresponds} Assume, there is a representation $\rho_0 \in \Rep^{\Gamma_1}_{G,\pi}(\OO)$. Under the isomorphism \eqref{eq_bij_Z1} the action \eqref{DM_action_defi}
corresponds to the ring homomorphism
$$ \OO[(\mathcal E\otimes M)_{\Delta}]\rightarrow \OO[(\mathcal E\otimes M)_{\Delta}]\otimes \OO[M],$$ 
which sends $x\in (\mathcal E\otimes M)_{\Delta}$ to $x\otimes q(x)$, where 
$q: (\mathcal E\otimes M)_{\Delta}\twoheadrightarrow I_{\Delta}M$ is the map of \eqref{exact_seq_nat}.
\end{prop}

\begin{proof}
    We have seen in the proof of \Cref{Z1_Rep}, that the action of $\DD(M)$ on $\Rep^{\Gamma_1}_{G,\pi}$ corresponds under the isomorphism \eqref{eq_bij_Z1} to the action via the boundary map $$\DD(M)(A) \to Z^1(\Gamma_1, \DD(M)(A))$$ and group multiplication. 
    By commutativity of the left square in \Cref{prop_represent} this map is induced by the composition $(\mathcal E \otimes M)_{\Delta} \xrightarrow{q} I_{\Delta}M \to M$. As the comultiplication on $\OO[(\mathcal E \otimes M)_{\Delta}]$ is given by $x \mapsto x \otimes x$ for all $x \in (\mathcal E \otimes M)_{\Delta}$, the claim follows by composing with $\OO[(\mathcal E\otimes M)_{\Delta}] \to \OO[M]$ in the second factor.
\end{proof}

\begin{cor}\label{free_action} Assume that there is a representation $\rho_0 \in \Rep^{\Gamma_1}_{G,\pi}(\OO)$. The action of $\DD(M)$ on $\Rep^{\Gamma_1}_{G,\pi}$ factors through the 
action of $\DD(I_{\Delta}M)$, which acts freely.
\end{cor}

\begin{proof}
    As in the proof of \Cref{action_corresponds} we use the composition $(\mathcal E \otimes M)_{\Delta} \xrightarrow{q} I_{\Delta}M \to M$ to see, that the boundary map $\DD(M)(A) \to Z^1(\Gamma_1, \DD(M)(A))$ factors through $\DD(I_{\Delta}M)$.
    
    Let $A$ be any $\OO$-algebra and let $H(A)$ be the group $\Hom( (\mathcal E \otimes M)_{\Delta}, A^{\times})$. Since $q: (\mathcal E \otimes M)_{\Delta} \rightarrow I_{\Delta}M$ is surjective, we may identify $\DD(I_{\Delta}M)(A)$ with a subgroup of $H(A)$.  
    The isomorphism \eqref{ORep_iso} 
    identifies $\Rep^{\Gamma_1}_{G,\pi}(A)$ with  $H(A)$
    and the action of $\DD(I_{\Delta}M)(A)$ on $\Rep^{\Gamma_1}_{G,\pi}(A)$
    is identified with the action of  $\DD(I_{\Delta}M)(A)$ on $H(A)$ by multiplication, which is a free action. Hence, the action of $\DD(I_{\Delta}M)$ on $\Rep^{\Gamma_1}_{G,\pi}$ is free.
\end{proof}

\begin{cor}\label{git_isos} Assume that there is a representation $\rho_0 \in \Rep^{\Gamma_1}_{G,\pi}(\OO)$. The isomorphisms \eqref{fpqc_sheafification} and \eqref{eq_bij_Z1} induce an isomorphism
\begin{align}
    \Rep^{\Gamma_1}_{G,\pi} /\DD(M) \eqto \Spec(\OO[H_1(\Gamma_1, M)]),\label{fpqc_quotient}
\end{align}
where $\Rep^{\Gamma_1}_{G,\pi} /\DD(M)$ denotes the fpqc sheafification of the presheaf quotient.
Moreover this quotient is a GIT quotient. In particular, \eqref{fpqc_quotient} induces isomorphisms
\begin{align}
    \OO[H_1(\Gamma_1, M)] \eqto \OO(\Rep^{\Gamma_1}_{G,\pi})^{G^0} \eqto \OO[(\mathcal E\otimes M)_{\Delta}]^{G^0} \label{GIT_isos}
\end{align}
and the composition of the maps in \eqref{GIT_isos} is induced by $\varphi_{\nat}$.
\end{cor}

\begin{proof}
    The isomorphism \eqref{fpqc_quotient} follows directly from \Cref{H1_sheafify} and \Cref{Z1_Rep}.
    To see that $\Rep^{\Gamma_1}_{G,\pi} /\DD(M)$ is also a GIT quotient, we observe that every $\DD(M)$-equivariant map from $\Rep^{\Gamma_1}_{G,\pi}$ to an affine scheme equipped with the trivial $\DD(M)$-action factors through $\Rep^{\Gamma_1}_{G,\pi} /\DD(M)$ by the universal property of the fpqc sheaf quotient.
    Since $\Rep^{\Gamma_1}_{G,\pi} /\DD(M)$ is affine it is indeed the GIT quotient and this gives the first isomorphism in \eqref{GIT_isos}.
    The second isomorphism in \eqref{GIT_isos} comes from \eqref{ORep_iso}.

    Let $B \in \OO\hyphen\alg$.
    After removing the $G^0$-invariants in \eqref{GIT_isos} we get two maps
    $$ \Hom((\mathcal E\otimes M)_{\Delta}, B^{\times}) \eqto \Rep_{G,\pi}^{\Gamma_1}(B) \to \Hom(H_1(\Gamma_1, M), B^{\times}) $$
    The first map comes from \eqref{Z1withE_iso} and \eqref{eq_bij_Z1}. The second map comes from \eqref{univ_coeff} and \eqref{eq_bij_Z1}.
    Hence the composition is \eqref{eq_nat} with $V = B^{\times}$, so the last assertion follows from \Cref{image_nat}.
\end{proof}

\begin{lem}\label{condensed} Let $\tau: G\hookrightarrow \mathbb A^n$ be a closed immersion 
of $\OO$-group schemes and let $\rho\in \Rep^{\Gamma_1}_{G, \pi}(A)$. 
Assume that $\Rep^{\Gamma_1}_{G, \pi}(\OO)$ is non-empty. 
Then $\tau(\rho(\Gamma_1))$ is contained in a finitely generated
$\OO[H_1(\Gamma_1, M)]$-submodule of $A^n = \mathbb A^n(A)$. 
\end{lem}

\begin{proof} Proposition \ref{Z1_Rep} implies that it is enough 
to prove the statement, when $A=\OO[(\mathcal E\otimes M)_{\Delta}]$
and 
$\rho= \Phi_{\nat}\rho_0$, where $\rho_0$ is any representation in
$\Rep^{\Gamma_1}_{G, \pi}(\OO)$ and $\Phi_{\nat}$ is the cocycle defined in Section \ref{sec_Z1}.

Since $\Gamma_2$ is of finite index in $\Gamma_1$ it is enough to show that $\tau(\rho(\Gamma_2))$ is contained in a finitely 
generated $\OO[H_1(\Gamma_1, M)]$-submodule of $(\OO[(\mathcal E\otimes M)_{\Delta}])^n$. Since $\Gamma_2$ acts trivially on $M$ we have 
a canonical isomorphism 
$$ H_1(\Gamma_2, M)\cong \Gamma_2^{\ab}\otimes M,$$
and it follows from \eqref{exact_seq_nat} that 
the image of $\OO[\Gamma_2^{\ab}\otimes M]$ 
in $\OO[(\mathcal E\otimes M)_{\Delta}]$ is contained in 
$\OO[H_1(\Gamma_1, M)]$. We thus may assume that $G=\DD(M)$ 
and $\Gamma_2=\Gamma_1$. In this case, 
$\OO[H_1(\Gamma, M)]= \OO[(\mathcal E\otimes M)_{\Delta}]$, so there
is nothing to prove. 
\end{proof}

\section{\texorpdfstring{Lafforgue's $G$-pseudocharacters}{Lafforgue's G-pseudocharacters}}\label{Laf}

We keep the notation of \Cref{sec_adm_rep}.
We now recall Lafforgue's notion of \emph{$G$-pseudocharacter} in the form of \cite[Definition 3.1]{quast}.
The difference to Lafforgue’s original definition \cite[Section 11]{Laf} is that we work over the base ring $\OO$ and allow $G$ to be disconnected.
The definition works for generalised reductive $\OO$-group schemes as defined in \cite{defG}, but this generality is not needed here.

\begin{defi}\label{LafPC} Let $A$ be a commutative $\OO$-algebra. A \emph{$G$-pseudocharacter} $\Theta$ of $\Gamma$ over $A$ is a sequence $(\Theta_n)_{n \geq 1}$ of $\OO$-algebra maps
$$\Theta_n : \OO[G^n]^{G^0} \to \mathrm{Map}(\Gamma^n,A)$$ for $n \geq 1$, satisfying the following conditions\footnote{Here $G$ acts on $G^n$ by $g \cdot (g_1, \dots, g_n) = (gg_1g^{-1}, \dots, gg_ng^{-1})$. This induces a rational action of $G$ on the affine coordinate ring $\OO[G^n]$ of $G^n$. The submodule $\OO[G^n]^{G^0} \subseteq \OO[G^n]$ is defined as the rational invariant module of the $G^0$-representation $\OO[G^n]$. It is an $\OO$-subalgebra, since $G$ acts by $\OO$-linear automorphisms.}:
\begin{enumerate}
    \item For each $n,m \geq 1$, each map $\zeta : \{1, \dots, m\} \to \{1, \dots,n\}$, $f \in \OO[G^m]^{G^0}$ and $\gamma_1, \dots, \gamma_n \in \Gamma$, we have
    $$ \Theta_n(f^{\zeta})(\gamma_1, \dots, \gamma_n) = \Theta_m(f)(\gamma_{\zeta(1)}, \dots, \gamma_{\zeta(m)}) $$
    where $f^{\zeta}(g_1, \dots, g_n) = f(g_{\zeta(1)}, \dots, g_{\zeta(m)})$.
    \item For each $n \geq 1$, for each $\gamma_1, \dots, \gamma_{n+1} \in \Gamma$ and each $f \in \OO[G^n]^{G^0}$, we have
    $$ \Theta_{n+1}(\hat f)(\gamma_1, \dots, \gamma_{n+1}) = \Theta_n(f)(\gamma_1, \dots, \gamma_n\gamma_{n+1}) $$
    where $\hat f(g_1, \dots, g_{n+1}) = f(g_1, \dots, g_ng_{n+1})$.
\end{enumerate}
\end{defi}

We denote the set of $G$-pseudocharacters of $\Gamma_1$ over $A$ by $\PC_G^{\Gamma_1}(A)$.
The functor $A \mapsto \PC_G^{\Gamma_1}(A)$ is representable by an affine $\OO$-scheme $\PC_G^{\Gamma_1}$ \cite[Theorem 3.20]{quast}.
When $\varphi : G \to H$ is a homomorphism of generalised tori over $\OO$, the induced maps $\varphi^*_n : \OO[H^n]^{H^0} \to \OO[G^n]^{G^0}$ give rise to an $H$-pseudocharacter $(\Theta_n \circ \varphi^*_n)_{n \geq 1}$. By analogy with the notation for representations we denote this $H$-pseudocharacter by $\varphi \circ \Theta$. Thus we also have an induced map $\PC_G^{\Gamma}(A) \to \PC_H^{\Gamma}(A)$.
It is easy to verify that specialisation along $f : A \to B$ commutes with composition with $\varphi$, i.e. $(\varphi \circ \Theta) \otimes_A B = \varphi \circ (\Theta \otimes_A B)$.

Recall that for every homomorphism $\rho : \Gamma_1 \to G(A)$, there is an associated $G$-pseudocharacter $\Theta_{\rho} \in \PC^{\Gamma_1}_G(A)$, which depends on $\rho$ only up to $G^0(A)$-conjugation. For $m \geq 1$, the homomorphism $(\Theta_{\rho})_m : \OO(G^m)^{G^0} \to \Map(\Gamma_1^m, A)$ is defined by
\begin{align}
    (\Theta_{\rho})_m(f)(\gamma_1, \dots, \gamma_m) \colonequals f(\rho(\gamma_1), \dots, \rho(\gamma_m)). \label{assoc_pc}
\end{align}
There is a natural $G^0(A)$-equivariant map
\begin{align}
    \Rep^{\Gamma_1}_G(A) \to \PC^{\Gamma_1}_G(A), \quad\rho \mapsto \Theta_{\rho} \label{assoc_pc_map},
\end{align}
which induces maps of $\OO$-schemes
\begin{equation}
    \Rep^{\Gamma_1}_G \to \PC^{\Gamma_1}_G \label{rep_pc_map}.
\end{equation}    
Since $G^0$-acts trivially on the target \eqref{rep_pc_map} factors as    
   \begin{equation} 
    \Rep^{\Gamma_1}_G \sslash G^0 \to \PC^{\Gamma_1}_G. \label{git_pc_map}
\end{equation}

If $\rho \in \Rep^{\Gamma_1}_{G,\pi}(A)$, then $\Theta_{\rho}$ is \emph{admissible} in the sense, that the functorial image of $\Theta_{\rho}$ under the map $G \to \underline\Delta$ maps to the $\underline \Delta$-pseudocharacter attached to $\pi$ in $\PC^{\Gamma_1}_{\underline \Delta}(A)$. We denote the set of admissible $G$-pseudocharacters by $\PC^{\Gamma_1}_{G,\pi}(A)$. The functor $A \mapsto \PC^{\Gamma_1}_{G,\pi}(A)$ is representable by an affine $\OO$-scheme $\PC^{\Gamma_1}_{G,\pi}$. We have natural maps
\begin{align}
    \Rep^{\Gamma_1}_{G, \pi} \to \PC^{\Gamma_1}_{G, \pi} \label{rep_pc_without_git} \\
    \Rep^{\Gamma_1}_{G, \pi} \sslash G^0 \to \PC^{\Gamma_1}_{G, \pi} \label{git_pc_pi_map}
\end{align}
as above.

We will show, that \eqref{git_pc_map} and \eqref{git_pc_pi_map} are isomorphisms. The result is specific to generalised tori\footnote{We don't expect \eqref{git_pc_map} and \eqref{git_pc_pi_map} to be isomorphisms when $G$ is an arbitrary generalised reductive group in the sense of \cite[Definition 2.3]{quast}. However, we don't know of an example where \eqref{git_pc_map} is not an isomorphism.} and might be of independent interest.

\begin{prop}\label{git_pc_iso}
    The map
    \begin{align}
        \OO(\PC^{\Gamma_1}_G) \to \OO(\Rep^{\Gamma_1}_G)^{G^0} \label{git_pc_map_alg}
    \end{align}
    corresponding to \eqref{git_pc_map} is an isomorphism.
\end{prop}

\begin{proof}
    As $G^0$ is linearly reductive the proof of \cite[Proposition 2.11 (i)]{emerson2023comparison} applies, we recall the argument here. Let $\OO(G^{\Gamma_1})$ be the $\OO$-algebra which represents the functor $$\OO\hyphen\alg \to \Set, ~A \mapsto \Map(\Gamma_1, G(A)).$$
    We have a natural surjection $\OO(G^{\Gamma_1}) \twoheadrightarrow \OO(\Rep^{\Gamma_1}_G)$ with kernel $J$.
    We also have a natural surjection $\OO(G^{\Gamma_1})^{G^0} \twoheadrightarrow \OO(\PC^{\Gamma_1}_G)$ with kernel $I$.
    The structure of $J$ and $I$ is described explicitly in \cite[Proposition 2.5]{emerson2023comparison}; note that there $G$ is assumed to be connected, but the description generalises easily to our situation. Namely $J$ is generated by the image of the $G^0$-equivariant maps 
    $$\varphi_{\gamma, \delta} : \OO(G \times G^{\Gamma_1}) \to \OO(G^{\Gamma_1})$$ 
    for all $\gamma, \delta \in \Gamma_1$, where 
    $$ \varphi_{\gamma, \delta}(f)((g_{\alpha})_{\alpha \in \Gamma_1}) = f(g_{\gamma\delta}, (g_{\alpha})_{\alpha \in \Gamma_1}) - f(g_{\gamma}g_{\delta}, (g_{\alpha})_{\alpha \in \Gamma_1}). $$
    The ideal $I$ is generated by the image of $\varphi_{\gamma, \delta}(\OO(G \times G^{\Gamma_1})^{G^0})$ for all $\gamma, \delta \in \Gamma_1$.

    Since taking $G^0$-invariants is exact, the natural map $\OO(G^{\Gamma_1})^{G^0}/I = \OO(\PC^{\Gamma_1}_G) \to \OO(\Rep^{\Gamma_1}_G)^{G^0} = \OO(G^{\Gamma_1})^{G^0}/J^{G^0}$ is surjective.
    So it remains to show, that $J^{G^0} \subseteq I$.
    Let $h \in J^G$ and write $h = \sum_{i=1}^n \varphi_{\gamma_i, \delta_i}(f_i)$ where $\gamma_i, \delta_i \in \Gamma_1$ and $f_i \in \OO(G \times G^{\Gamma_1})$.
    Denote the Reynolds operator on $G^0$-modules by $E$, it commutes with the $G^0$-equivariant maps $\varphi_{\gamma_i, \delta_i}$, so that we have
    $$ h = E(h) = \sum_{i=1}^n \varphi_{\gamma_i, \delta_i}(E(f_i)) \in I. $$
\end{proof}

\begin{lem}\label{pc_is_rep} If the conjugation action of $G^0$ on all $\OO(G^m)$ is trivial then the following hold:
  \begin{enumerate}
  \item \eqref{rep_pc_map} is an isomorphism;
  \item if $\Gamma_1$ is a topological group and $A$ is a topological $\OO$-algebra, then \eqref{assoc_pc_map} induces a bijection between continuous representations and continuous $G$-pseudocharacters;
  \end{enumerate} 
  In particular,  if $G = G^0$ or $G^0$ is trivial then \eqref{rep_pc_map} is an isomorphism and hence $\Rep^{\Gamma_1}_{\underline\Delta} \cong \PC^{\Gamma_1}_{\underline\Delta}$.
\end{lem}

\begin{proof}
    Choose $\OO$-algebra generators $f_1, \dots, f_r \in \OO(G)$ and let $\Theta \in \PC^{\Gamma_1}_G(A)$.
    The functions $\Theta_1(f_1), \dots, \Theta_1(f_r)$ define a unique map $\rho : \Gamma_1 \to G(A)$ such that ${\Theta_1(f_1) = f_i \circ \rho}$ for all $i=1, \dots, r$.
    Let $\mu : \OO(G) \to \OO(G) \otimes_{\OO} \OO(G)$ be the comultiplication map. We can write
    $$ \mu(f_i) = \sum_{j,k} a_{ijk} f_j \otimes f_k $$
    for some $a_{ijk} \in \OO$.
    By (2) of \Cref{LafPC}, we have
    \begin{align*}
        f_i(\rho(\gamma_1\gamma_2)) &= \Theta_1(f_i)(\gamma_1\gamma_2) = \Theta_2(\mu(f_i))(\gamma_1, \gamma_2) \\
        &= \sum_{j,k} a_{ijk} \Theta_1(f_j)(\gamma_1) \Theta_1(f_k)(\gamma_2) \\
        &= \sum_{j,k} a_{ijk} f_j(\rho(\gamma_1)) f_k(\rho(\gamma_2)) \\
        &= \mu(f_i)(\rho(\gamma_1), \rho(\gamma_2)) = f_i(\rho(\gamma_1) \rho(\gamma_2)),
    \end{align*}
    so $\rho$ is a homomorphism.
    For every $m \geq 1$, every function $f \in \OO(G^m)$ can be written as a linear combination of functions of the form $(g_1, \dots, g_m) \mapsto f_i(g_j)$.
    So by rule (1) of \Cref{LafPC} the function $\Theta_m(f)$ is determined by $\Theta_1(f_1), \dots, \Theta_1(f_r)$.
    It follows, that $\Theta$ is the only $G$-pseudocharacter satisfying $\Theta_1(f_i) = f_i \circ \rho$ for all $i$.
    By definition $\Theta = \Theta_{\rho}$, so this shows surjectivity and injectivity of \eqref{rep_pc_map}.

    For the claim about continuity we observe, that continuity of $\Theta_1(f_1), \dots, \Theta_1(f_r)$ is equivalent to continuity of $\rho$.
\end{proof}

\begin{cor}\label{cor_PC_inv}
    The map
    \begin{align}
        \OO(\PC^{\Gamma_1}_{G,\pi}) \to \OO(\Rep^{\Gamma_1}_{G,\pi})^{G^0} \label{rep_pc_pi_map_alg}
    \end{align}
    corresponding to \eqref{git_pc_pi_map} is an isomorphism.
\end{cor}

\begin{proof}
    Via the canonical projection $\Pi : G \to \underline\Delta$ we have a natural $G^0$-action on $\OO(\Rep^{\Gamma_1}_{\underline\Delta})$, which is trivial, since $\underline\Delta^0 = 1$.
    So the natural $G^0$-equivariant map $\OO(\Rep^{\Gamma_1}_{\underline\Delta}) \to \OO(\Rep^{\Gamma_1}_G)$ induces a map $\OO(\Rep^{\Gamma_1}_{\underline\Delta}) \to \OO(\Rep^{\Gamma_1}_G)^{G^0}$.
    The representation $\pi : \Gamma_1 \to \Delta$ corresponds to a map $\OO(\Rep^{\Gamma_1}_{\underline\Delta}) \to \OO$.
    The tensor product $\OO(\Rep^{\Gamma_1}_G) \otimes_{\OO(\Rep^{\Gamma_1}_{\underline\Delta})} \OO$ represents $\Rep^{\Gamma_1}_{G,\pi}$.
    By virtue of \Cref{right_exactness_G0} we have a natural isomorphism
    \begin{align*}
        \OO(\Rep^{\Gamma_1}_G)^{G^0} \otimes_{\OO(\Rep^{\Gamma_1}_{\underline\Delta})} \OO \eqto \OO(\Rep^{\Gamma_1}_{G,\pi})^{G^0}.
    \end{align*}
    On the other hand $\OO(\PC^{\Gamma_1}_{G,\pi})$ is representable by
    $$\OO(\PC^{\Gamma_1}_G) \otimes_{\OO(\PC^{\Gamma_1}_{\underline\Delta})} \OO \cong \OO(\PC^{\Gamma_1}_G) \otimes_{\OO(\Rep^{\Gamma_1}_{\underline\Delta})} \OO$$
    by \Cref{pc_is_rep}. We conclude, that \eqref{rep_pc_pi_map_alg} is obtained from \eqref{git_pc_map_alg} by applying $- \otimes_{\OO(\Rep^{\Gamma_1}_{\underline\Delta})} \OO$. So by \Cref{git_pc_iso} \eqref{rep_pc_pi_map_alg} is an isomorphism as well.
\end{proof}

\begin{cor}\label{GIT_isos_2}
    Assume, there is a representation $\rho_0 \in \Rep^{\Gamma_1}_{G,\pi}(\OO)$.
    Then we have isomorphisms $\OO(\PC^{\Gamma_1}_{G,\pi}) \cong \OO[H_1(\Gamma_1, M)] \cong \OO[(\mathcal E\otimes M)_{\Delta}]^{G^0}$ via \eqref{GIT_isos} and \eqref{rep_pc_pi_map_alg}.
\end{cor}

\begin{lem}\label{comes_from_rep}
    Assume, that $\Rep^{\Gamma_1}_{G,\pi}(\OO)$ is non-empty.
    Then the map $\OO(\PC^{\Gamma_1}_{G,\pi}) \to \OO(\Rep^{\Gamma_1}_{G,\pi})$ corresponding to \eqref{rep_pc_without_git} has a section as $\OO$-algebras.
    In particular, for all $A \in \OO\hyphen\alg$ and all ${\Theta \in \PC^{\Gamma_1}_{G,\pi}(A)}$, there is a representation $\rho \in \Rep^{\Gamma_1}_{G, \pi}(A)$, such that $\Theta = \Theta_{\rho}$.
\end{lem}

\begin{proof}
    Let $\rho_0 \in \Rep^{\Gamma_1}_{G,\pi}(\OO)$.
    In the diagram
    \begin{center}
        \begin{tikzcd}[column sep=small]
            \OO(\PC^{\Gamma_1}_{G,\pi}) \arrow[r, "\eqref{rep_pc_pi_map_alg}"] & \OO(\Rep^{\Gamma_1}_{G,\pi})^{G^0} \arrow[r, "\sim"] \arrow[d, hookrightarrow] & \OO[(\mathcal E \otimes M)_{\Delta}]^{G^0} \arrow[d, hookrightarrow] & \OO[H_1(\Gamma_1, M)] \arrow[d, hookrightarrow] \arrow[l, swap, "\eqref{GIT_isos}"] \\
            & \OO(\Rep^{\Gamma_1}_{G,\pi}) \arrow[r, "\eqref{ORep_iso}"] \arrow[ul, "\sigma", dashed, bend left=20, twoheadrightarrow] & \OO[(\mathcal E \otimes M)_{\Delta}] & \OO[H_1(\Gamma_1, M)][t_1^{\pm 1}, \dots, t_s^{\pm 1}] \arrow[u, swap, dashed, bend right=40, twoheadrightarrow] \arrow[l, swap, "\eqref{poly_over_H1}"]
        \end{tikzcd}
    \end{center}    
    the left solid square commutes, since \eqref{ORep_iso} is a $G^0$-equivariant isomorphism. The right solid square commutes, since \eqref{GIT_isos} is by \Cref{cor_PC_inv} induced by $\varphi_{\nat}$. The section of the right vertical arrow is defined by mapping $t_1, \dots, t_s$ to $1$ and since all horizontal maps are isomorphisms defines a section $\sigma$ of the left vertical arrow.
    Given a $G$-pseudocharacter $\Theta \in \PC^{\Gamma_1}_{G,\pi}(A)$, it corresponds to a homomorphism $\OO(\PC^{\Gamma_1}_{G,\pi}) \to A$ and by restriction along $\sigma$, we get a representation $\rho \in \Rep^{\Gamma_1}_{G,\pi}(A)$, which by restriction along \eqref{rep_pc_pi_map_alg} recovers $\Theta$, so $\Theta = \Theta_{\rho}$.
\end{proof}

\section{Profinite completion} \label{sec_profinite}

So far we have worked with representations and pseudocharacters of 
an abstract group $\Gamma_1$ and topology did not play a role. In 
this section we transfer some of the results proved in 
Sections \ref{sec_adm_rep} and \ref{Laf} to the  profinite completion $\widehat{\Gamma}_1$ of $\Gamma_1$ and impose continuity conditions
on the representations and pseudocharacters of $\widehat{\Gamma}_1$.

Let $\OO$ be the ring of integers in a finite extension of $L$ with 
uniformiser $\varpi$ and residue field $k$. 
\subsection{Deformations of representations} Let $\rhobar: \widehat{\Gamma}_1\rightarrow G(k)$ 
be a continuous representation, such that $\rhobar|_{\Gamma_1}\in \Rep^{\Gamma_1}_{G,\pi}(k)$. 

Let $D^{\square}_{\rhobar}:\Aa_{\OO}\rightarrow \Set$ be the functor such that 
$D^{\square}_{\rhobar}(A)$ is the set of continuous representations 
$\rho_A: \widehat{\Gamma}_1\rightarrow G(A)$ such that $\rho_A\equiv \rhobar\pmod{\mm_A}$. 
The functor $D^{\square}_{\rhobar}$ is pro-representable by $R^{\square}_{\rhobar}\in \widehat{\Aa}_{\OO}$.

\begin{lem}\label{limit_square} There is a natural isomorphism in $\widehat{\Aa}_{\OO}$:
\begin{equation}\label{eq_limit}
 R^{\square}_{\rhobar}\cong \varprojlim_I \OO(\Rep^{\Gamma_1}_G)_{\mm}/I,
\end{equation} 
where $\mm$ is the maximal ideal of $\OO(\Rep^{\Gamma_1}_G)$ corresponding 
to $\rhobar$ and the limit is taken over all the ideals $I$, such that the 
quotient is finite.
\end{lem}
\begin{proof} If $A\in \Aa_{\OO}$ and $\rho\in D^{\square}_{\rhobar}(A)$ then 
$\rho|_{\Gamma_1}\in \Rep^{\Gamma_1}_{G, \pi}(A)$ and hence we obtain a natural 
homomorphism of $\OO$-algebras $\OO(\Rep^{\Gamma_1}_G)\rightarrow A$. The 
condition $\rho\equiv \rhobar\pmod{\mm_A}$ implies that this map 
extends to the localisation $\OO(\Rep^{\Gamma_1}_G)_{\mm}\rightarrow A$.
Since $A$ is finite the kernel of this map is one of the ideals $I$ appearing 
in the inductive system in \eqref{eq_limit}.

Conversely, if $I$ is an ideal of $\OO(\Rep^{\Gamma_1}_G)_{\mm}$ such that 
the quotient $A$ is finite then $A\in \Aa_{\OO}$. The map 
$\OO(\Rep^{\Gamma_1}_G)\rightarrow k$ corresponding to $\rhobar$ 
factors through $\OO(\Rep^{\Gamma_1}_{G, \pi})\rightarrow k$ 
as $\rhobar|_{\Gamma_1}\in \Rep^{\Gamma_1}_{G,\pi}(k)$ by assumption.
Since 
$G/G^0$ is constant and $A$ is a local ring we have $(G/G^0)(A)=(G/G^0)(k)$ and hence the map $\OO(\Rep^{\Gamma_1}_G)\rightarrow A$ factors through $\OO(\Rep^{\Gamma_1}_{G, \pi})\rightarrow A$. 
By specialising the 
universal representation $\Gamma_1\rightarrow G(\OO(\Rep^{\Gamma_1}_{G, \pi}))$ 
along this map we obtain a representation
$\rho: \Gamma_1\rightarrow G(A)$. Since the target is finite $\rho$
extends uniquely to a continuous representation $\rhotilde:\widehat{\Gamma}_1\rightarrow G(A)$, which defines 
a point in $D^{\square}_{\rhobar}(A)$.

Hence, the right hand side of \eqref{eq_limit} pro-represents $D^{\square}_{\rhobar}$ and the assertion follows.

\end{proof}

\begin{lem}\label{abstract_limit} Let $\mathcal A$ be an abelian group and let $\mm$ be a 
maximal ideal of the group ring $\OO[\mathcal A]$ with residue field 
$k$. Then there is an isomorphism in $\widehat{\Aa}_{\OO}$: 
\begin{equation}\label{eq_abs_limit}
\OO\br{N}\cong  \varprojlim_I \OO[\mathcal A]_{\mm}/I,
\end{equation}
where $N$ is the pro-$p$ completion of $\mathcal A$ and the limit is taken over all the ideals $I$, such that the quotient is finite.
\end{lem}

\begin{proof} Let $\psibar: \mathcal A \rightarrow k^{\times}$ be the 
character obtained by the composition $\mathcal A\rightarrow \OO[\mathcal A]/\mm=k$ and let $\psi: \mathcal A\rightarrow \OO^{\times}$ be any character lifting $\psibar$ (for example the Teichm\"uller lift). Then $\mm$ is generated by 
    $\varpi$ and $( \psi(a)^{-1} a - 1)$ for all $a\in \mathcal A$. 
    The map $\varphi: \OO[\mathcal A]\rightarrow \OO[\mathcal A]$, 
    which sends $a\in \mathcal A$ to $\psi(a) a$ is an isomorphism of 
    $\OO$-algebras, which sends $\mm$ to the maximal ideal corresponding to the trivial character. We thus may assume that $\mm$ is generated by $\varpi$
    and the augmentation ideal of $\OO[\mathcal A]$.

    If $P$ is a finite $p$-power order quotient of $\mathcal A$ then $\OO/\varpi^n[P]$ 
    is a finite local ring and the surjection $\OO[\mathcal A]\twoheadrightarrow \OO/\varpi^n[P]$ maps 
    $\mm$ to the maximal ideal of $\OO/\varpi^n[P]$. Hence, the map factors through 
    $\OO[\mathcal A]_{\mm}\twoheadrightarrow \OO/\varpi^n[P]$. Conversely, if $A$ is a quotient of $\OO[\mathcal A]_{\mm}$, which 
    is finite (as a set), then the image of $\mathcal A$ is contained in $1+\mm_A$, and 
    hence is a finite group of $p$-power order. This implies the assertion.
\end{proof}

\begin{lem}\label{noeth_comp} Let $\mathcal A$ be an abelian group and let $\mm$ be a 
maximal ideal of the group ring $\OO[\mathcal A]$ with residue field 
$k$ and let $N$ be the pro-$p$ completion of $\mathcal A$. Then the 
following are equivalent:
\begin{enumerate}
\item $N$ is a finitely generated $\Zp$-module;
\item $\mathcal A/ p \mathcal A$ is a finite $p$-group;
\item $\OO\br{N}$ is noetherian.
\end{enumerate}
If the equivalent conditions hold then $\OO\br{N}$ is naturally isomorphic to the $\mm$-adic completion 
of $\OO[\mathcal A]$.
\end{lem}
\begin{proof} The equivalence of (1) and (2) follows from Nakayama's lemma
and the fact that $N/pN$ is the pro-$p$ completion of $\mathcal A/p \mathcal A\cong \oplus_I \Fp$ for some set $I$.

If $N$ is finitely generated as a $\Zp$-module, then $N\cong \mu\oplus \Zp^r$, 
where $\mu$ is a finite $p$-group and $r=\dim_{\Qp} N\otimes_{\Zp} \Qp$. 
Hence, $\OO\br{N}\cong \OO[\mu]\br{x_1, \ldots, x_r}$ and hence (1) implies (3).

If $N$ is not finitely generated as a $\Zp$-module then we may 
find a strictly increasing nested family of finitely generated 
$\Zp$-submodules $N_i \subset N$. Since they are finitely generated 
they are closed in $N$. The kernels of $\OO\br{N}\twoheadrightarrow \OO\br{N/N_i}$
form a strictly increasing nested family of ideals in $\OO\br{N}$. Hence, 
(3) implies (1).

As in the proof of Lemma \ref{abstract_limit} we may assume that 
$\mm$ is generated by $\varpi$ and the augmentation ideal of $\OO[\mathcal A]$. 
The map $\mathcal A\rightarrow \mm$, $a\mapsto a-1$ induces an isomorphism 
of $k$-vector spaces $k\otimes_{\ZZ} \mathcal A \cong \mm/(\varpi, \mm^2)$. 
If (2) holds then $\dim_k \mm/(\varpi, \mm^2)$ is finite. Hence,  the powers 
    of $\mm$ form a cofinal system in the inverse system in \eqref{eq_abs_limit} and the last assertion follows from 
    Lemma \ref{abstract_limit}.
\end{proof}

\begin{lem}\label{completion_def_square} If $\Gamma_2^{\ab}/p \Gamma_2^{\ab}$ is a finite $p$-group 
and $\Rep^{\Gamma_1}_{G, \pi}(\OO)$ is non-empty then 
$R^{\square}_{\rhobar}$ is naturally isomorphic to the $\mm$-adic completion 
of $\OO(\Rep^{\Gamma_1}_{G, \pi})$, where $\mm$ is the maximal ideal 
corresponding to $\rhobar|_{\Gamma_1}$. 
\end{lem} 
\begin{proof} Proposition \ref{Z1_Rep} allows us to identify 
$\OO(\Rep^{\Gamma_1}_{G,\pi})$ with $\OO[(\mathcal E\otimes M)_{\Delta}]$. 
It follows from \eqref{Eextension} that if $\Gamma_2^{\ab}/p \Gamma_2^{\ab}$
is finite then $\mathcal E/p \mathcal E$ is also finite. Since $M$ is a finite free $\ZZ$-module, the assertion 
follows from \Cref{limit_square} and Lemma \ref{noeth_comp} applied to $\mathcal A=(\mathcal E\otimes M)_{\Delta}$.
\end{proof} 

\begin{lem} \label{over_k}
There is a natural isomorphism of local $k$-algebras between 
$R^{\square}_{\rhobar}/\varpi$ and the completed group algebra $k\br{((\mathcal E\otimes M)_{\Delta})^{\wdp}.}$.
\end{lem}
\begin{proof} The ring $R^{\square}_{\rhobar}/\varpi$ represents 
the restriction of $D^{\square}_{\rhobar}$ to $\Aa_k$. Lemma \ref{limit_square}
implies that 
\begin{equation}\label{eq1}
R^{\square}_{\rhobar}/\varpi\cong \varprojlim_I \OO(\Rep^{\Gamma_1}_G)_{\mm}/(\varpi,I).
\end{equation}
Since $\rhobar|_{\Gamma_1}\in \Rep^{\Gamma_1}_{G, \pi}(k)$ we may apply Proposition \ref{Z1_Rep}
with the base ring $\OO$ equal to $k$ to obtain 
\begin{equation}\label{eq2}
\OO(\Rep^{\Gamma_1}_{G,\pi})/\varpi \cong k[(\mathcal E\otimes M)_{\Delta}].
\end{equation}
The assertion then follows from \eqref{eq1}, \eqref{eq2} and Lemma \ref{abstract_limit}.
\end{proof}

Let $\Gamma\twoheadrightarrow \widehat{\Gamma}_1$ be a surjection of 
profinite groups, such that for all $A\in \Aa_{\OO}$ every  
continuous representation $\rho_A: \Gamma\rightarrow G(A)$ factors 
through $\widehat{\Gamma}_1$. 

Let $\ad \rhobar$ be the representation of $\Gamma$ on 
$\Lie G_k$ 
obtained by composing $\rhobar$ with the adjoint representation. 
We write $H^i_{\cont}(\Gamma, \ad \rhobar)$ for the $i$-th continuous
group cohomology and denote by $h^i_{\cont}(\Gamma, \ad\rhobar)$ 
its dimension as a $k$-vector space. 

\begin{prop}\label{lci} If $h^1_{\cont}(\Gamma, \ad\rhobar)$ is 
finite then $((\mathcal E\otimes M)_{\Delta})^{\wdp}$ is a finitely generated $\Zp$-module. Moreover, 
if $h^2_{\cont}(\Gamma, \ad\rhobar)$ is also finite and 
\begin{equation}\label{inequality}
h^1_{\cont}(\Gamma, \ad\rhobar) -h^0_{\cont}(\Gamma, \ad\rhobar)-
h^2_{\cont}(\Gamma, \ad\rhobar) \ge \rank_{\Zp} ((\mathcal E\otimes M)_{\Delta})^{\wdp} -\dim G_k
\end{equation}
then $R^{\square}_{\rhobar}$ is complete intersection, $\OO$-flat
of relative dimension $\rank_{\Zp} ((\mathcal E\otimes M)_{\Delta})^{\wdp}$.
\end{prop}

\begin{proof} The assumption on $\Gamma$ implies that we could have
equivalently defined the deformation problem $D^{\square}_{\rhobar}$ 
with $\Gamma$ instead of $\widehat\Gamma_1$. 

 By standard obstruction theory due to Mazur, we have a presentation 
\begin{equation}\label{present_again}
\frac{\OO\br{x_1, \ldots, x_r}}{(f_1, \ldots, f_s)}\cong R^{\square}_{\rhobar},
\end{equation}
where $r=\dim_k Z^1_{\cont}(\Gamma, \ad \rhobar)$ and $s=h^2_{\cont}(\Gamma, \ad\rhobar)$.
The exact sequence
$$0\rightarrow (\ad \rhobar)^{\Gamma}\rightarrow \ad \rhobar
\rightarrow Z^1_{\cont}(\Gamma, \ad \rhobar)\rightarrow H^1_{\cont}(\Gamma, 
\ad \rhobar)\rightarrow 0$$
implies that $r= h^1_{\cont}(\Gamma, \ad\rhobar) -h^0_{\cont}(\Gamma, \ad\rhobar)+\dim G_k$. Let $N:=((\mathcal E\otimes M)_{\Delta})^{\wdp}$. By considering \eqref{present_again} modulo $\varpi$
and using Lemma \ref{over_k} we deduce that $\rank_{\Zp} N\ge r -s$. Moreover, 
the assumption \eqref{inequality} implies that $r-s \ge \rank_{\Zp} N$. Hence,  $r-s = \rank_{\Zp} N$
and $\varpi, f_1, \ldots, f_s$ can be extended to a system of 
parameters of a regular ring $\OO\br{x_1, \ldots, x_r}$. Thus $\varpi, f_1, \ldots, f_s$ are a part of a regular sequence, and hence $R^{\square}_{\rhobar}$ is complete intersection, flat over $\OO$  of 
relative dimension of $\rank_{\Zp} N$. 
\end{proof} 

\begin{cor}\label{7_7} If the assumptions of \Cref{lci} hold then there is
an isomorphism of local $\OO'$-algebras:
$$ \OO'\otimes_{\OO} R^{\square}_{\rhobar}\cong \OO'\br{((\mathcal E\otimes M)_{\Delta})^{\wdp}},$$
where $\OO'$ is the ring of integers in  a finite extension of $L$.
In particular, $\Rep^{\widehat\Gamma_1}_{G,\pi}(\OO')$ and $\Rep^{\Gamma_1}_{G,\pi}(\OO')$ are  non-empty.
\end{cor}
\begin{proof} It follows from Proposition \ref{lci} that there
is a homomorphism of $\OO$-algebras $x:R^{\square}_{\rhobar}\rightarrow \Qpbar$.
Then $\kappa(x)$ is a finite extension of $L$ and the image of 
$x$ is contained in the ring of integers of $\kappa(x)$, which we denote 
by $\OO'$. Let $\rho_x: \widehat{\Gamma}_1\rightarrow G(\OO')$ be the specialisation of the universal deformation along $x: R^{\square}_{\rhobar}\rightarrow \OO'$. Then $\rho_x\in \Rep^{\widehat\Gamma_1}_{G,\pi}(\OO')$ and its restriction 
to $\Gamma_1$ defines a point in $\Rep^{\Gamma_1}_G(\OO')$.

Since $\OO'\otimes_{\OO} R^{\square}_{\rhobar}$ represents the 
functor $D^{\square}_{\rhobar_{k'}}: \Aa_{\OO'}\rightarrow \Set$
where $\rhobar_{k'}$ is the composition of $\rhobar$
with $G(k)\hookrightarrow G(k')$, where $k'$ is the residue field 
of $\OO'$, we may assume that $\OO'=\OO$.

Proposition \ref{Z1_Rep} allows us to identify 
$\OO(\Rep^{\Gamma_1}_{G,\pi})$ with $\OO[(\mathcal E\otimes M)_{\Delta}]$.
This identification depends on the lift $\rho_x$. The assertion follows
from Lemmas \ref{limit_square} and \ref{abstract_limit}.
\end{proof}

\subsection{Deformations of pseudocharacters} If $A$ is a topological ring then a $G$-pseudocharacter $\Theta \in \PC_{G,\pi}^{\widehat{\Gamma}_1}(A)$ is \emph{continuous},  if every $\Theta_m$ takes values in the subset $\mathcal C(\widehat{\Gamma}_1^m, A) \subseteq \Map(\widehat{\Gamma}_1^m, A)$ of continuous maps.
We write $\cPC^{\widehat{\Gamma}_1}_{G,\pi}(A) \subseteq \PC^{\widehat{\Gamma}_1}_{G,\pi}(A)$ for the subset of continuous $G$-pseudocharacters.

\begin{lem}\label{cPC_PC_bijection}
    Assume, that $\Rep^{\Gamma_1}_{G,\pi}(\OO)$ is non-empty.
    Let $A$ be a finite discrete $\OO$-algebra.
    Then the natural map $\cPC^{\widehat{\Gamma}_1}_{G,\pi}(A) \to \PC^{\Gamma_1}_{G,\pi}(A), \Theta\mapsto \Theta|_{\Gamma_1}$ is bijective.
\end{lem}

\begin{proof}
    As the image of $\Gamma_1$ in $\widehat{\Gamma}_1$ is dense, injectivity follows from \cite[Lemma 3.2]{quast}.
    Let $\Theta \in \PC^{\Gamma_1}_{G,\pi}(A)$.
    By \Cref{comes_from_rep} there is a representation $\rho \in \Rep^{\Gamma_1}_{G,\pi}(A)$, such that $\Theta = \Theta_{\rho}$.
    Since $A$ is finite, $\rho$ extends to a continuous representation $\tilde\rho : \widehat\Gamma_1 \to G(A)$ and we have $\Theta_{\tilde\rho}|_{\Gamma_1} = \Theta$.
\end{proof}

Let $\Thetabar$ be the $G$-pseudocharacter of $\widehat{\Gamma}_1$ associated to $\rhobar$. Let $D^{\ps}_{\Thetabar}: \Aa_{\OO}\rightarrow \Set$ be the deformation functor which sends $A$ to the 
set of continuous $G$-pseudocharacters $\Theta$ of $\widehat{\Gamma}_1$ valued in $A$ such that 
$\Theta\otimes_A k= \Thetabar.$ The functor $D^{\ps}_{\Thetabar}$ is pro-represented by $R^{\ps}_{\Thetabar}\in \widehat{\Aa}_{\OO}$ by 
\cite[Theorem 5.4]{quast}.

\begin{prop} \label{prop_Rps}
    Assume, that $\Rep^{\Gamma_1}_{G,\pi}(\OO)$ is non-empty. Then there is 
    an isomorphism in $\widehat{\Aa}_{\OO}$:
    \begin{equation}\label{defi_R}
R^{\ps}_{\Thetabar} \cong \varprojlim_I \OO(\PC^{\Gamma_1}_{G,\pi})_{\mm}/ I,
\end{equation} 
where $\mm$ is the maximal ideal of $\OO(\PC^{\Gamma_1}_{G,\pi})$ corresponding 
to $\Thetabar|_{\Gamma_1}$, and the limit is taken over all ideals $I$ such that 
the quotient is finite. 
    
\end{prop}

\begin{proof} If $A \in \Aa_{\OO}$ then we  have natural bijections
    \begin{align*}
        \Hom_{\widehat{\Aa}_{\OO}}(R^{\ps}_{\Thetabar}, A) \cong \{\Theta \in \cPC^{\widehat\Gamma_1}_{G, \pi}(A) \mid \Theta \otimes_A k = \overline{\Theta}\},\\
        \Hom_{\text{local }\OO\hyphen\alg}(\OO(\PC^{\Gamma_1}_{G,\pi})_{\mm}, A) \cong \{\Theta \in \PC^{\Gamma_1}_{G, \pi}(A) \mid \Theta \otimes_A k = \overline{\Theta}\}.
    \end{align*}
    As $\cPC^{\widehat{\Gamma}_1}_{G, \pi}(A) \cong \PC^{\Gamma_1}_{G, \pi}(A)$ by \Cref{cPC_PC_bijection}, the claim follows.
\end{proof}

\begin{lem}\label{no_name} Assume, that $\Rep^{\Gamma_1}_{G,\pi}(\OO)$ is non-empty. Then $R^{\ps}_{\Thetabar}$ is isomorphic to the completed group algebra $\OO\br{H_1(\Gamma_1, M)^{\wdp}}$.
\end{lem}

\begin{proof} Since $\Rep^{\Gamma_1}_{G,\pi}(\OO)$ is non-empty Corollary \ref{GIT_isos_2} 
allows us to identify $\OO(\PC^{\Gamma_1}_{G,\pi})$ with  $\OO[H_1(\Gamma, M)]$.
The assertion follows from Proposition \ref{prop_Rps} and Lemma \ref{abstract_limit}.
\end{proof} 

\begin{lem}\label{Rps_madic} If $\Gamma_2^{\ab}/p \Gamma_2^{\ab}$ is a finite $p$-group and 
$\Rep^{\Gamma_1}_{G,\pi}(\OO)$ is non-empty
then $R^{\ps}_{\Thetabar}$ is noetherian and is naturally isomorphic 
to the $\mm$-adic completion of $\OO(\PC^{\Gamma_1}_{G,\pi})$, where $\mm$ is the maximal ideal of $\OO(\PC^{\Gamma_1}_{G,\pi})$ corresponding 
to $\Thetabar|_{\Gamma_1}$.
\end{lem}

\begin{proof} The exact sequence of $\Gamma_1$-modules $0\rightarrow I_{\Delta}\otimes M \rightarrow \ZZ[\Delta]\otimes M \rightarrow M\rightarrow 0$ 
induces an exact sequence in homology:
$$H_1(\Gamma_2, M)\rightarrow H_1(\Gamma_1, M)\rightarrow (I_{\Delta}\otimes M)_{\Delta}.$$
Since the action of $\Gamma_2$ on $M$ is trivial we have a canonical 
isomorphism $H_1(\Gamma_2, M)\cong \Gamma_2^{\ab}\otimes M$. Since 
$I_{\Delta}$ and $M$ are free $\ZZ$-modules of finite rank we conclude that 
the assumption $\Gamma_2^{\ab}/p \Gamma_2^{\ab}$ is finite implies 
that $\mathcal A/p \mathcal A$ is finite, where $\mathcal A=H_1(\Gamma_1, M)$. 
The assertion follows from Lemmas \ref{no_name} and \ref{noeth_comp}.
\end{proof}

\subsection{Moduli space of representations} We study the following schemes. 
\begin{defi} For $\mathcal G= \Gamma_1$ or $\widehat{\Gamma}_1$, let 
$X^{\gen, \mathcal G}_{G, \Thetabar}: R^{\ps}_{\Thetabar}\text{-}\alg\rightarrow \Set$ be the functor 
$$X^{\gen,\mathcal G}_{G,\Thetabar}(A):=\{ \rho\in \Rep^{\mathcal G}_{G, \pi}(A): \Theta_\rho = \Theta^u\otimes_{R^{\ps}_{\Thetabar}} A\},$$
where $\Theta^u\in D^{\ps}_{\Thetabar}(R^{\ps}_{\Thetabar})$ is the universal deformation of $\Thetabar$ and we consider its restriction to $\Gamma_1$ if $\mathcal G=\Gamma_1$.
\end{defi}

\begin{prop}\label{paris} Assume that $\Rep^{\widehat{\Gamma}_1}_{G,\pi}(\OO)$ is non-empty. The restriction to $\Gamma_1$ induces an isomorphism
$$X^{\gen,\widehat{\Gamma}_1}_{G,\Thetabar} \overset{\cong}{\longrightarrow}
X^{\gen,\Gamma_1}_{G,\Thetabar}.$$
In particular, $X^{\gen,\widehat{\Gamma}_1}_{G,\Thetabar}$ is representable
by the $R^{\ps}_{\Thetabar}$-algebra isomorphic to 
$$R^{\ps}_{\Thetabar}\otimes_{\OO(\PC^{\Gamma_1}_{G, \pi})} 
\OO(\Rep^{\Gamma_1}_{G, \pi})\cong R^{\ps}_{\Thetabar}[t_1^{\pm 1}, \ldots, t_s^{\pm 1}],$$
where $s= \rank_{\ZZ} M - \rank_{\ZZ} M_{\Delta}$.
\end{prop}
\begin{proof} We let $\mathcal G$ be either $\Gamma_1$ or its profinite 
completion $\widehat{\Gamma}_1$ in the proof. The functor $X^{\gen,\mathcal G}_{G,\Thetabar}$ is representable 
by $R^{\ps}_{\Thetabar}\otimes_{\OO(\PC^{\mathcal G}_{G, \pi})} 
\OO(\Rep^{\mathcal G}_{G, \pi})$. Let $\rho_0 \in \Rep^{\widehat{\Gamma}_1}_{G,\pi}(\OO)$ then $\rho_0|_{\Gamma_1}$ is 
in $\Rep^{\Gamma_1}_{G,\pi}(\OO)$, and using these representations 
we may identify $\Rep^{\mathcal G}_{G,\pi}$ with the space of 
$1$-cocycles $Z^1(\mathcal G, \DD(M)(-))$ by Proposition \ref{Z1_Rep}.
Corollaries \ref{M_ffr} and \ref{GIT_isos_2} imply that after choosing a basis of 
$I_{\Delta}M$ as a $\ZZ$-module we may identify 
$$ \OO(\Rep^{\mathcal G}_{G,\pi})\cong \OO(\PC^{\mathcal G}_{G, \pi})[t_1^{\pm 1}, \ldots, t_s^{\pm 1}],$$
thus for both $\mathcal G= \Gamma_1$ and $\mathcal G=\widehat{\Gamma}_1$ we have
\begin{equation}\label{Rps_curly_G}
R^{\ps}_{\Thetabar}\otimes_{\OO(\PC^{\mathcal G}_{G, \pi})} 
\OO(\Rep^{\mathcal G}_{G, \pi})\cong R^{\ps}_{\Thetabar}[t_1^{\pm 1}, \ldots, t_s^{\pm 1}]
\end{equation}
and under theses isomorphisms the restriction to $\Gamma_1$ is just the identity map. 
\end{proof}

The following Lemma will allow us to relate the scheme $X^{\gen, \widehat{\Gamma}_1}_{G, \Thetabar}$ to the scheme $X^{\gen}_{G, \rhobarss}$ 
introduced in \cite{defG}. 

\begin{lem}\label{defG1} Let $\tau: G\hookrightarrow \mathbb A^n$ be a closed immersion 
of $\OO$-schemes, let $A$ be an $R^{\ps}_{\Thetabar}$-algebra  and let $\rho\in X^{\gen, \widehat{\Gamma}_1}_{G, \Thetabar}(A)$. 
Assume that $\Rep^{\widehat\Gamma_1}_{G, \pi}(\OO)$ is non-empty. 
Then $\tau(\rho(\widehat{\Gamma}_1))$ is contained in a finitely generated
$R^{\ps}_{\Thetabar}$-submodule of $A^n = \mathbb A^n(A)$. 
\end{lem}
\begin{proof} Since $X^{\gen, \widehat{\Gamma}_1}_{G, \Thetabar}$ 
is represented by $R^{\ps}_{\Thetabar}\otimes_{\OO(\PC^{\widehat{\Gamma}_1}_{G, \pi})} 
\OO(\Rep^{\widehat \Gamma_1}_{G, \pi})$ the assertion follows from Lemma \ref{condensed} applied to $\widehat{\Gamma}_1$, and \Cref{GIT_isos_2}.
\end{proof}

\subsection{Irreducible components} \label{sec_irr_comp}

Let $N$ be a finitely generated $\Zp$-module and let $\OO\br{N}$ be the completed group algebra of $N$, where $\OO$ is  the ring of integers in a finite extension $L$ of $\Qp$. 
Let $\fDN: \widehat{\Aa}_{\OO} \rightarrow \Ab$ be a formal group scheme, given by 
\begin{equation}\label{pts_fDN}
\fDN(A)\colonequals \Hom_{\widehat{\Aa}_{\OO}}(\OO\br{N}, A)= \Hom^{\cont}_{\Group}(N, A^{\times}).
\end{equation}
Multiplication in $\fDN$ is induced by the map 
\begin{equation}\label{comult}
c: \OO\br{N}\rightarrow \OO\br{N}\wtimes_{\OO} \OO\br{N}, \quad n\mapsto n\wtimes n , \quad \forall n\in N.
\end{equation}
Let $\mu$ be the torsion subgroup of $N$. Then we have a non-canonical isomorphism $N\cong \mu \oplus \Zp^r$, where $r=\dim_{\Qp} N\otimes_{\Zp} \Qp$. 
This induces an isomorphism 
\begin{equation}\label{non_can_iso}
\OO\br{N}\cong \OO[\mu]\br{x_1, \ldots, x_r}.
\end{equation}
We assume that $L$ contains all the $p^m$-th roots of unity, where $p^m$ is the order of $\mu$. Then the group $\mathrm X(\mu)$  of characters
$\chi: \mu\rightarrow \OO^{\times}$ also has order $p^m$. The following Lemma 
is an immediate consequence of  \eqref{non_can_iso}: 
\begin{lem}\label{easy_lem} The following hold:
\begin{enumerate}
\item the irreducible components of $\Spec \OO\br{N}$ are in canonical bijection with $\mathrm X(\mu)$, so that 
the irreducible component corresponding to $\chi\in\mathrm X(\mu)$ is given by $\Spec ( \OO\br{N}\otimes_{\OO[\mu], \chi}\OO)$;
\item every irreducible component of $\Spec \OO\br{N}$ contains an $\OO$-rational point $\psi\in \fDN(\OO)$;
\item a point $\psi\in \fDN(\OO)=\Hom^{\cont}_{\Group}(N, \OO^{\times})$ lies 
on the irreducible component corresponding to $\chi\in\mathrm X(\mu)$ if and only if $\psi(x)=\chi(x)$ for all $x\in \mu$. 
\end{enumerate}
\end{lem}

Let $R$ be a complete local noetherian $\OO$-algebra with residue field $k$. Let ${\mathcal X:\widehat{\Aa}_{\OO} \rightarrow \Set}$ be the functor 
$\mathcal X(A)= \Hom_{\widehat{\Aa}_{\OO}}(R, A)$. This functor is represented 
by a formal scheme $\Spf R$. Let us suppose that we have a faithful and 
transitive action of $\fDN$ on $\mathcal X$. Concretely, this means that for 
all $A\in \widehat{\Aa}_{\OO}$ we have a faithful and transitive  action of 
the group $\fDN(A)$ on the set $\mathcal X(A)$, which is functorial in $A$. It is enough to 
restrict to $A\in \Aa_{\OO}$ as the general case follows by continuity. 
The action map $\fDN\times \mathcal X \rightarrow \mathcal X$ gives us 
a morphism in $\widehat{\Aa}_{\OO}$: 
\begin{equation}\label{action_map} 
\alpha: R \rightarrow \OO\br{N}\wtimes_{\OO} R.
\end{equation}
Let us assume that $\mathcal X(\OO)$ is non-empty and choose $x\in \mathcal X(\OO)$. Since the action map is faithful and transitive for every $A\in \widehat{\Aa}_{\OO}$ the map 
\begin{equation}\label{bijective_action}
\fDN(A) \rightarrow \mathcal X(A), \quad \psi\mapsto \psi\cdot x_A
\end{equation} 
is bijective, where $x_A$ is the image of $x$ in $\mathcal X(A)$. This implies 
that the composition 
$$ \alpha_x: R \overset{\alpha}{\longrightarrow} \OO\br{N}\wtimes_{\OO} R\rightarrow (\OO\br{N}\wtimes_{\OO} R)\wtimes_{R, x} \OO\cong \OO\br{N}$$
is an isomorphism. 
\begin{examp} If $\mathcal X=\fDN$ and the action is given by left translations, 
then it follows from \eqref{pts_fDN}, \eqref{comult} that $\alpha_{\psi}$ corresponds 
to the character $N \rightarrow \OO\br{N}^{\times}$, $n \mapsto \psi(n)n$.
\end{examp}

\begin{lem}\label{pts_R_comp} Assume that $\mathcal X(\OO)$ is non-empty. Then the following hold:
\begin{enumerate} 
\item every irreducible component of $\Spec R$ has an $\OO$-rational point $x\in \mathcal X(\OO)$;
\item every $x\in \mathcal X(\OO)$ lies on a unique irreducible component of $\Spec R$;
\item if  $x, y\in \mathcal X(\OO)$ then there exists a unique $\psi\in \fDN(\OO)$ such that $\psi\cdot x = y$; 
\item let  $x$, $y$ and $\psi$ be as in part $(3)$. Then $x$, $y$ lie on the same irreducible component of $\Spec R$ 
if and only if $\psi|_{\mu} =1$.
\end{enumerate}
\end{lem}
\begin{proof} If $\mathcal X= \fDN$ with the action given by 
left translations then the assertions follow readily from \Cref{easy_lem}. 

In the general case, we pick $x\in \mathcal X(\OO)$. 
The isomorphism $\alpha_x: R\rightarrow \OO\br{N}$ of $\OO$-algebras 
allows us to reduce the question to the case above. For parts (3) and (4), we note that the bijection \eqref{bijective_action} is $\fDN(A)$-equivariant 
for the action by left translations on the source. 
\end{proof}

\begin{lem}\label{lem_irr_comp} Assume that $\mathcal X(\OO)$ is non-empty. Then the action of $\fDN(\OO)$ on $\mathcal X(\OO)$ induces
a canonical action of $\mathrm X(\mu)$ on the set of irreducible components of $\Spec R$. This action is 
faithful and transitive. 
\end{lem} 
\begin{proof} Since $\mu$ is a direct summand of $N$ the map $\psi\mapsto \psi|_{\mu}$ induces a
surjective group homomorphism $\fDN(\OO)\twoheadrightarrow \mathrm X(\mu)$. Let $K$ be the kernel of this map. 
\Cref{pts_R_comp} implies that there is a natural bijection between the set of irreducible components 
of $\Spec R$ and the set of $K$-orbits in $\mathcal X(\OO)$. The action $\fDN(\OO)$ on the set of 
$K$-orbits factors through the action of $\mathrm X(\mu)$, which induces the sought after action on the set of 
irreducible components. Since the action of $\fDN(\OO)$ on $\mathcal X(\OO)$ is faithful and transitive the same applies 
to the action of $\mathrm X(\mu)$. 
\end{proof}

We will now  get back to our particular example. Let 
$\Zhat^1: \Aa_{\OO}\rightarrow \Ab$
be the functor such that
$\Zhat^1(A)$ is the set of continuous $1$-cocycles $\Phi:\widehat{\Gamma}_1\rightarrow \Hom(M, 1+\mm_A)$ for the discrete topology on the target. 
\begin{lem} The functor $\Zhat^1$ is pro-represented by 
$\OO\br{((\mathcal E\otimes M)_{\Delta})^{\wdp}}$. 
\end{lem}
\begin{proof} Let $G=\DD(M)\rtimes \underline{\Delta}$,  $\rhobar: \widehat{\Gamma}_1\rightarrow G(k), \gamma \mapsto (1, \pi(\gamma))$ and let $\rho_0: \widehat{\Gamma}_1\rightarrow G(\OO), 
\gamma\mapsto (1, \pi(\gamma))$. The map $\Phi \mapsto \Phi\rho_0$ induces a natural bijection between $\Zhat^1(A)$
and $D^{\square}_{\rhobar}(A)$ for all $A\in \Aa_{\OO}$; this can be shown by the same argument as in the proof of \Cref{Z1_Rep}.
Thus $\Zhat^1$ is pro-represented by $R^{\square}_{\rhobar}$ and the assertion follows from
\Cref{Z1_Rep},  \Cref{limit_square}  and \Cref{abstract_limit}.
\end{proof}

It follows from Proposition \ref{Z1_Rep} that if $\Phi\in \Zhat^1(A)$ and $\rho\in D^{\square}_{\rhobar}(A)$ 
then $\gamma\mapsto \Phi(\gamma) \rho(\gamma)$ defines 
a representation $\Phi\rho\in D^{\square}_{\rhobar}(A)$ and 
the map $\Zhat^1\times D^{\square}_{\rhobar}\rightarrow D^{\square}_{\rhobar}$, $(\Phi, \rho) \mapsto \Phi\rho$ 
defines a faithful and transitive action of $\Zhat^1$ on $D^{\square}_{\rhobar}$.

\begin{prop}\label{components1} Assume that $\Gamma_2^{\ab}/p \Gamma_2^{\ab}$ is finite and let  $\mu$ be the torsion subgroup of $((\mathcal E\otimes M)_{\Delta})^{\wdp}$.
Assume further that $\Rep^{\Gamma_1}_{G, \pi}(\OO)$ is non-empty and 
$\OO$ contains all the $p^m$-th roots of unity, where $p^m$ is the order of $\mu$.
Then there is a canonical action of the character group $\mathrm X(\mu)$ on 
the set of irreducible components of $\Spec R^{\square}_{\rhobar}$, 
$\Spec R^{\ps}_{\Thetabar}$ and $X^{\gen}_{G, \Thetabar}$, respectively. 
Moreover, the following hold:
\begin{enumerate}
\item this action is faithful and transitive;
\item irreducible components of $\Spec R^{\square}_{\rhobar}$ and 
$\Spec R^{\ps}_{\Thetabar}$ are formally smooth over $\OO$;
\item irreducible components of $X^{\gen}_{G, \Thetabar}$ are 
of the form $$\Spec \OO\br{x_1, \ldots, x_r}[t_1^{\pm 1}, \ldots, t_s^{\pm 1}],$$ 
where $r=\rank_{\Zp} H_1(\Gamma_1, M)^{\wdp}$ and $r+s=\rank_{\Zp} ((\mathcal E\otimes M)_{\Delta})^{\wdp}$.
\end{enumerate}
\end{prop}
\begin{proof} The assumption that $\Gamma_2^{\ab}/p \Gamma_2^{\ab}$ is finite implies 
that $((\mathcal E\otimes M)_{\Delta})^{\wdp}$ is a finitely generated 
$\Zp$-module by \Cref{noeth_comp}, and hence $\mu$ is a finite $p$-group.

The assumption that $\Rep^{\Gamma_1}_{G, \pi}(\OO)$ is non-empty implies 
via Proposition \ref{Z1_Rep} that $\OO(\Rep^{\Gamma_1}_{G, \pi})\cong 
\OO[(\mathcal E\otimes M)_{\Delta}]$. It follows from Lemmas \ref{limit_square},,
and \ref{abstract_limit} that $R^{\square}_{\rhobar}\cong \OO\br{((\mathcal E\otimes M)_{\Delta})^{\wdp}}$. In particular, $D^{\square}_{\rhobar}(\OO)$
is non-empty. It follows from Lemma \ref{lem_irr_comp} that the action 
of $\Zhat^1(\OO)$ on $D^{\square}_{\rhobar}(\OO)$ induces 
a transitive and faithful action of $\mathrm X(\mu)$ on the set of 
irreducible components of $R^{\square}_{\rhobar}$.

Let us spell out Lemma \ref{pts_R_comp} in our context. Given an irreducible component $X$ of $\Spec R^{\square}_{\rhobar}$, $X(\OO)$ is non-empty and 
we pick any $\rho\in X(\OO)$; given $\chi\in \mathrm X(\mu)$ we pick 
any $\Phi\in \Zhat^1(\OO)$ such that $\Phi(\gamma)=\chi(\gamma)$ for all
$\gamma\in \mu$. Then $\chi\cdot X$ is the unique irreducible component 
of $\Spec R^{\square}_{\rhobar}$ such that $\Phi \rho\in (\chi\cdot X)(\OO)$. 

It follows from \Cref{prop_represent}
that 
$$((\mathcal E\otimes M)_{\Delta})^{\wdp}\cong H_1(\Gamma_1, M)^{\wdp}\oplus 
\Zp^s$$ for some $s\ge 0$. Hence, $R^{\square}_{\rhobar}\cong R^{\ps}_{\Thetabar}\br{t_1, \ldots, t_s}$ and $\mu$ is the torsion subgroup of 
$H_1(\Gamma_1, M)^{\wdp}$. Thus the map $\Spec R^{\square}_{\rhobar} \rightarrow 
\Spec R^{\ps}_{\Thetabar}$ is $\mathrm X(\mu)$-equivariant and induces
an $\mathrm X(\mu)$-equivariant bijection between the irreducible components. 

Since $A^{\gen}_{G, \Thetabar}\cong R^{\ps}_{\Thetabar}[t_1^{\pm 1}, \ldots, 
t_s^{\pm 1}]$ by \Cref{paris} the map $X^{\gen}_{G, \Thetabar}\rightarrow 
\Spec R^{\ps}_{\Thetabar}$ is $\mathrm X(\mu)$-equivariant and induces 
an $\mathrm X(\mu)$-equivariant bijection  between the sets of 
irreducible components.

The isomorphism $R^{\ps}_{\Thetabar} \cong \OO\br{H_1(\Gamma, M)^{\wdp}}$ 
allows us to consider $R^{\ps}_{\Thetabar}$ as an $\OO[\mu]$-algebra. This 
is non-canonical: it amounts to distinguishing one irreducible component
of $\Spec R^{\ps}_{\Thetabar}$. Once we do this the other irreducible
are given by $R^{\ps}_{\Thetabar}\otimes_{\OO[\mu],\chi} \OO$ for 
$\chi\in \mathrm X(\mu)$, and the special component corresponds to the trivial 
character. These are isomorphic to 
$\OO\br{H_1(\Gamma, M)^{\wdp}/\mu}$ and hence are formally smooth. 

Similarly irreducible components of $\Spec R^{\square}_{\rhobar}$ 
and $X^{\gen}_{G, \Thetabar}$ are given by 
$$R^{\square}_{\rhobar}\otimes_{\OO[\mu], \chi}\OO\cong \OO\br{((\mathcal E\otimes M)_{\Delta})^{\wdp}/\mu}$$
and 
$$A^{\gen}_{G, \Thetabar}\otimes_{\OO[\mu], \chi}\OO\cong 
\OO\br{H_1(\Gamma, M)^{\wdp}/\mu}[t_1^{\pm 1}, \ldots, t_s^{\pm 1}],$$
respectively. 
\end{proof}

\section{Rank calculations}\label{sec_rank}

Let $E$ be a finite Galois extension of $F$, let $\Delta\colonequals\Gal(E/F)$ and 
let $M$ be a free $\ZZ$-module of finite rank with an action of $\Delta$. 
In this section we compute the $\Zp$-rank and the torsion subgroup of the 
pro-$p$ completion of $(E^{\times}\otimes M)^{\Delta}$ and related modules. 
These calculations are used in the next section. 

If $\mathcal A$ is an abelian group 
we denote its pro-$p$ completion by $\mathcal A^{\wedge, p}$. If $N$ is a $\Zp$-module 
we define $\rank_{\Zp} N:= \dim_{\Qp} (N\otimes_{\Zp} \Qp)$. 

Let $\Gamma_E^{\ab}$ be the maximal abelian pro-finite quotient of 
$\Gamma_E$ and let $\Gamma_E^{\abp}$ be the maximal abelian pro-$p$ 
quotient of $\Gamma_E$. The Artin map $\Art_E: E^{\times}\rightarrow \Gamma_E^{\ab}$ induces an isomorphism between the profinite completion 
of $E^{\times}$ and $\Gamma_E^{\ab}$. Thus $(E^{\times}\otimes M)^{\wdp}\cong 
\Gamma_E^{\abp}\otimes M$.

\begin{lem}\label{tate_finite} Let $N$ be a finitely generated $\ZZ[\Delta]$
(resp. $\Zp[\Delta]$) module. Then the Tate 
cohomology groups $\widehat{H}^i( \Delta, N)$
are finite for all $i\in \ZZ$.
\end{lem}

\begin{proof} If $N$ is finitely generated over $\ZZ$ then the statement is proved in \cite[Corollary 2, p. 105]{CasselsFroehlich}. The same argument carries over if 
$N$ is finitely generated over $\Zp$: the cohomology
groups are finitely generated $\Zp$-modules, since the complex 
computing cohomology 
consists of finitely
generated $\Zp$-modules, moreover they are annihilated by the order
of $\Delta$.
\end{proof}

\begin{lem}\label{lem_rank_one} $\rank_{\Zp} (M^{\wdp} \otimes I_{\Delta})_{\Delta} = \rank_{\ZZ} M - \rank_{\ZZ} M_{\Delta}.$
\end{lem}

\begin{proof} The long exact sequence in homology attached to
$$0 \rightarrow M^{\wdp}\otimes I_{\Delta}\rightarrow M^{\wdp}\otimes \ZZ[\Delta]\rightarrow M^{\wdp}\rightarrow 0$$
yields an exact sequence 
$$H_1(\Delta, M^{\wdp})\rightarrow (M^{\wdp}\otimes I_{\Delta})_{\Delta} \rightarrow M^{\wdp} \rightarrow (M^{\wdp})_{\Delta}\rightarrow 0,$$
as $M \otimes \ZZ[\Delta]$ is induced by the projection formula. Since we may swap coinvariants with completions and $M_{\Delta}$ is finitely generated we have
$$ \rank_{\Zp} (M^{\wdp})_{\Delta} = \rank_{\Zp} (M_{\Delta})^{\wdp}=\rank_{\ZZ} M_{\Delta}.$$
Since $M^{\wdp}$ is finitely generated over $\Zp$ and $\Delta$ is finite,  the group is $H_1(\Delta, M^{\wdp})$ is finite by Lemma \ref{tate_finite}. This implies the assertion. 
\end{proof}

\begin{lem}\label{tate_coinv_inv} $\rank_{\Zp} (\Gamma_E^{\ab,p}\otimes M)_{\Delta}=\rank_{\Zp} ((E^{\times}\otimes M)^{\Delta})^{\wedge,p}.$
\end{lem}

\begin{proof} 
    Since $M$ is a free $\ZZ$-module, we have an exact sequence of $\Delta$-modules 
    $$0\rightarrow (1+\pp_E)\otimes M \rightarrow E^{\times}\otimes M \rightarrow (E^{\times}/(1+\pp_E))\otimes M\rightarrow 0.$$
    Since $1+\pp_E$ is a finitely generated $\Zp$-module and $E^{\times}/(1+\pp_E)$ is a finitely generated 
    $\ZZ$-module, we deduce that $\widehat{H}^i(\Delta,  E^{\times}\otimes M)$ are finite for all $i\in \ZZ$.
    From the exact sequence: 
    $$ 0\rightarrow \widehat{H}^{-1}(\Delta, E^{\times}\otimes M)\rightarrow 
    (E^{\times}\otimes M)_{\Delta} \rightarrow (E^{\times}\otimes M)^{\Delta}\rightarrow 
    \widehat{H}^0(\Delta, E^{\times}\otimes M)\rightarrow 0$$
    we deduce that the pro-$p$ completions of $(E^{\times}\otimes M)_{\Delta}$ and of 
    $(E^{\times}\otimes M)^{\Delta}$ have the same $\Zp$-rank. 
    It follows from the universal property of pro-$p$ completion, that it commutes with taking $\Delta$-coinvariants. Hence the pro-$p$ completion 
    of $(E^{\times}\otimes M)_{\Delta}$ is isomorphic to $((E^{\times})^{\wdp}\otimes M)_{\Delta}\cong (\Gamma_E^{\abp}\otimes M)_{\Delta}$.
\end{proof}

\begin{lem}\label{lem_rank_two} $\rank_{\ZZ_p} (\Gamma_E^{\abp}\otimes M)_{\Delta} = \rank_{\ZZ} M \cdot [F:\QQ_p] + \rank_{\ZZ} M_{\Delta}$.
\end{lem}

\begin{proof}
    We consider the long exact sequence in homology
    $$ H_1(\Delta, M^{\wdp}) \to ((\OO_E^{\times})^{\wdp} \otimes M)_{\Delta} \to ((E^{\times})^{\wdp} \otimes M)_{\Delta} \to (M^{\wdp})_{\Delta} \to 0.$$
    Since completion commutes with taking $\Delta$-coinvariants we have 
    $$\rank_{\Zp} (M^{\wdp})_{\Delta}=\rank_{\ZZ} M_{\Delta}.$$ 
    Since $H_1(\Delta, M^{\wdp})$ is finite by \Cref{tate_finite}, we are left to compute the $\Zp$-rank of  $((\OO_E^{\times})^{\wdp} \otimes M)_{\Delta}$. We note that $(\OO_E^{\times})^{\wdp}$ is equal to $1+\pp_E$ and another application of \Cref{tate_finite} shows that the rank does not change 
    if we replace $(\OO_E^{\times})^{\wdp}$ with any open $\Delta$-stable subgroup $V$ of $1+\pp_E$. 
    
    We choose $V$ to be  the image of a $p$-adic exponential function defined on $\pp_E^n$ for some large enough $n \geq 1$. We then have an isomorphism $\pp_E^n \otimes M \cong V \otimes M$ of $\Delta$-modules. 
    Since $\pp_E^n \cong \OO_E$ is isomorphic to $\OO_F[\Delta]$ as $\ZZ_p[\Delta]$-module, see the proof of \cite[Section 1.4]{serre_brighton}, $\OO_E \otimes M$ is free and thus $(\OO_E \otimes M)_{\Delta} \cong \OO_F \otimes M$. Thus 
    $\rank_{\Zp} ((\OO_E^{\times})^{\wdp} \otimes M)_{\Delta}= [F:\Qp] \cdot \rank_{\ZZ} M$ and the assertion follows. 
\end{proof}

\begin{lem}\label{torsion} The torsion subgroup of 
$((E^{\times}\otimes M)^{\Delta})^{\wdp}$ is equal to $(\mu_{p^{\infty}}(E)\otimes M)^{\Delta}$. 
\end{lem}

\begin{proof}
    The image of  $(E^{\times}\otimes M)^{\Delta}\rightarrow(E^{\times}/\OO_E^{\times})\otimes M$ is a free $\ZZ$-module of finite rank, which we denote by $s$, as the target is a free $\ZZ$-module of finite rank. The kernel of this map is equal to $(\OO_E^{\times}\otimes M)^{\Delta}$. 
    The Teichm\"uller lift gives an isomorphism of $\Delta$-modules 
    $\OO_{E}^{\times} \cong (1+\pp_E) \oplus k_E^{\times}$ and hence we have an isomorphism of abelian groups
    \begin{equation}\label{duck}
    (E^{\times}\otimes M)^{\Delta}\cong ((1+\pp_E)\otimes M)^{\Delta} \oplus(k_E^{\times}\otimes M)^{\Delta} \oplus \ZZ^s.
    \end{equation}
     Since the order of $k_E^{\times}$ is prime to $p$ and 
    $((1+\pp_E)\otimes M)^{\Delta}$ is closed in $(1+\pp_E)\otimes M$ and hence $p$-adically complete,
    we conclude that the torsion subgroup in the pro-$p$ completion of 
    $(E^{\times}\otimes M)^{\Delta}$ coincides with the torsion subgroup in 
    $((1+\pp_E)\otimes M)^{\Delta}$, which is equal to $(\mu_{p^{\infty}}(E)\otimes M)^{\Delta}$.
\end{proof}

\begin{lem}\label{finite_swap} If $\mathcal A$ is a finitely generated $\ZZ[\Delta]$-module 
then $(\mathcal A^{\wdp})^{\Delta}\cong (\mathcal A^{\Delta})^{\wdp}$.
\end{lem}
\begin{proof} We have $\mathcal A^{\wdp}\cong \mathcal A\otimes \Zp$ and
$(\mathcal A^{\Delta})^{\wdp}\cong \mathcal A^{\Delta}\otimes \Zp$ by 
\cite[\href{https://stacks.math.columbia.edu/tag/00MA}{Tag 00MA}]{stacks-project}.
We may express 
$\mathcal A^{\Delta}$ as the kernel of 
$$\mathcal A\rightarrow \bigoplus_{\delta\in \Delta} \mathcal A, \quad a\mapsto (\delta a -a)_{\delta\in \Delta}.$$
Since $\Zp$ is a flat $\ZZ$-module we conclude that 
$\mathcal A^{\Delta}\otimes \Zp \cong (\mathcal A\otimes \Zp)^{\Delta}.$
\end{proof}
\begin{lem}\label{duck_duck} $((E^{\times}\otimes M)^{\Delta})^{\wdp}\cong (\Gamma_E^{\abp}\otimes M)^{\Delta}$.
\end{lem} 
\begin{proof} Let $n_0$ be an integer such that $\exp: \pp_E^n \rightarrow 1+\pp_E^n$ converges for all $n\ge n_0$ and let $V_n:= \exp(\pp_E^n)$. Then 
$V_n$ for $n\ge n_0$ form a basis of open neighbourhoods of $1$ in $1+\pp_E$. 
Since $\pp_E^n\cong \OO_E \cong \OO_F[\Delta]$ as $\Delta$-module and 
$\exp$ is $\Delta$-equivariant we have an isomorphism 
$V_n\otimes M \cong \Ind^{\Delta}_{1} (\OO_F \otimes M)$ and hence 
$H^1(\Delta, V_n\otimes M)=0$. Thus for all $n\ge n_0$ we obtain an exact 
sequence 
\begin{equation}\label{boing}
0\rightarrow (V_n\otimes M)^{\Delta} \rightarrow (E^{\times}\otimes M)^{\Delta}
\rightarrow ((E^{\times}/ V_n)\otimes M)^{\Delta}\rightarrow 0.
\end{equation}
The completion and $\varprojlim_n$ are both limits and hence commute with each other. Lemma \ref{finite_swap} and \eqref{boing} imply that
\begin{equation}\label{bng}
((E^{\times}\otimes M)^{\Delta})^{\wdp}\cong \varprojlim_n ((E^{\times}/ V_n)\otimes M)^{\Delta})^{\wdp}\cong \varprojlim_n ((E^{\times}/ V_n)\otimes M)^{\wdp})^{\Delta}.
\end{equation}
The isomorphism $(E^{\times}\otimes M)^{\wdp}\cong \Gamma_E^{\abp}\otimes M$ 
induces an isomorphism
\begin{equation}\label{boing_boing}
((E^\times/V_n)\otimes M)^{\wdp}\cong (\Gamma_E^{\abp}/\Art_E(V_n))\otimes M.
\end{equation}
Since $H^1(\Delta, V_n \otimes M)=0$ we have an exact sequence
\begin{equation}
0\rightarrow (V_n\otimes M)^{\Delta} \rightarrow (\Gamma_E^{\abp}\otimes M)^{\Delta}
\rightarrow ((\Gamma_E^{\abp}/ \Art_E(V_n))\otimes M)^{\Delta}\rightarrow 0.
\end{equation}
We thus have 
$$(\Gamma^{\abp}_E \otimes M)^{\Delta}\cong \varprojlim_n ((\Gamma_E^{\abp}/ \Art_E(V_n))\otimes M)^{\Delta}$$
and the assertion follows from \eqref{boing_boing} and \eqref{bng}.
\end{proof}

\begin{cor}\label{rank_tors} There is an isomorphism of $\Zp$-modules:
$$(\Gamma_E^{\abp}\otimes M)^{\Delta}\cong (\mu_{p^{\infty}}(E)\otimes M)^{\Delta} \times \Zp^r,$$
where $r=\rank_{\ZZ} M \cdot [F:\Qp] + \rank_{\ZZ} M_{\Delta}$.
\end{cor}
\begin{proof} This follows from Lemmas \ref{duck_duck}, \ref{torsion}, \ref{lem_rank_two}, \ref{tate_coinv_inv}.
\end{proof}

\begin{prop}\label{main_rank} Let $0 \rightarrow E^{\times}\rightarrow \mathcal E\rightarrow I_{\Delta}\rightarrow 0$ be any extension of $\ZZ[\Delta]$-modules.
\begin{equation}\label{eq_rank_one}
\rank_{\Zp} ((\mathcal E\otimes M)_{\Delta})^{\wedge,p}= ([F:\Qp]+1) \rank_\ZZ M.
\end{equation}
\end{prop}

\begin{proof} Since $I_{\Delta}\otimes M$ is a free $\ZZ$-module the surjection 
$\mathcal E\otimes M\twoheadrightarrow I_{\Delta}\otimes M$ has a section, and hence we have an exact 
sequence of $\Zp[\Delta]$-modules 
$$0\rightarrow (E^{\times}\otimes M)^{\wedge, p}\rightarrow  (\mathcal E\otimes M)^{\wedge,p} 
\rightarrow (I_{\Delta}\otimes M)^{\wedge,p}\rightarrow 0.$$
Since $M$ is a free $\ZZ$-module of finite rank  and $(E^{\times})^{\wdp}\cong \Gamma^{\ab, p}_E$ we have
$$(E^{\times}\otimes M)^{\wdp}\cong (E^{\times})^{\wdp}\otimes M\cong \Gamma^{\ab,p}_E\otimes M.$$
Since $I_{\Delta}$ is a free $\ZZ$-module  $(I_{\Delta}\otimes M)^{\wdp}\cong I_{\Delta}\otimes M^{\wdp}$, so that we obtain an 
exact sequence of $\Zp[\Delta]$-modules: 
$$ 0\rightarrow \Gamma_E^{\ab,p}\otimes M \rightarrow (\mathcal E\otimes M)^{\wdp}\rightarrow I_{\Delta}\otimes M^{\wdp}\rightarrow 0.$$
Taking $\Delta$-coinvariants and observing that $H_1(\Delta, I_{\Delta}\otimes M^{\wdp})\cong 
H_2(\Delta, M^{\wdp})$ we obtain an exact sequence:
\begin{equation}\label{long_exact}
H_2(\Delta, M^{\wdp})\rightarrow (\Gamma_E^{\ab, p} \otimes M)_{\Delta}\rightarrow 
((\mathcal E\otimes M)^{\wdp})_{\Delta} \rightarrow (I_{\Delta}\otimes M^{\wdp})_{\Delta}\rightarrow 0.
\end{equation}
Since $H_2(\Delta, M^{\wdp})$ is a torsion module by Lemma \ref{tate_finite}, we deduce that 
$$\rank_{\Zp} ((\mathcal E\otimes M)^{\wdp})_{\Delta} = \rank_{\Zp} (\Gamma_E^{\ab, p} \otimes M)_{\Delta}
+ \rank_{\Zp} (I_{\Delta}\otimes M^{\wdp})_{\Delta}.$$
The assertion \eqref{eq_rank_one} follows from Lemmas \ref{lem_rank_one} and \ref{lem_rank_two}.
\end{proof}

\section{Galois deformations} 
\label{sec_gal_def}

We will apply the machinery developed in the previous sections in an arithmetic situation. Let $F$ be a finite extension of $\Qp$. We fix an algebraic 
closure $\overline F$ and let $\Gamma_F:= \Gal(\overline F/F)$. 
Let $\rhobar: \Gamma_F \rightarrow G(k)$ be a continuous representation, 
where $G$ is a generalised torus over $\OO$ and let $\Thetabar$ be the 
$G$-pseudocharacter associated to $\rhobar$. Let $\Gamma_E$ be the kernel 
of the composition $\Gamma_F \overset{\rhobar}{\longrightarrow} G(k)\rightarrow 
(G/G^0)(k)$ and let $\Delta:=\Gal(E/F)$. Let $\pi: \Gamma_F \rightarrow \Delta$ and $\Pi: G\rightarrow G/G^0$ 
be the projection maps. 

We let $D^{\ps}_{\Thetabar}: \Aa_{\OO} \rightarrow \Set$ be the functor 
$$D^{\ps}_{\Thetabar}(A)=\{\Theta\in \cPC^{\Gamma_F}_{G,\pi}(A): \Theta\otimes_A k= \Thetabar\}$$
and let $R^{\ps}_{\Thetabar}\in \widehat{\Aa}_{\OO}$ be the 
ring pro-representing $D^{\ps}_{\Thetabar}$ and let $\Theta^u$ be the universal 
deformation of $\Thetabar$. 

We let $X^{\gen}_{G, \Thetabar}: R^{\ps}_{\Thetabar}\text{-}\alg\rightarrow \Set$
be the functor $$X^{\gen}_{G, \Thetabar}(A)=\{ \rho\in \Rep^{\Gamma_F}_{G, \pi}(A): \Theta_{\rho}=\Theta^u\otimes_{R^{\ps}_{\Thetabar}} A\}.$$ Let 
$A^{\gen}_{G, \Thetabar}$ be the $R^{\ps}_{\Thetabar}$-algebra representing 
$X^{\gen}_{G, \Thetabar}$.

We let $D^{\square}_{\rhobar}: \Aa_{\OO} \rightarrow \Set$ be the functor
such that $D^{\square}_{\rhobar}(A)$ is the set of continuous 
representations $\rho: \Gamma_F\rightarrow G(A)$ such that $\rho\equiv \rhobar
\mod{\mm_A}$. Let $R^{\square}_{\rhobar}\in\widehat{\Aa}_{\OO}$ be the 
the ring pro-representing $D^{\square}_{\rhobar}$.

The Weil group $W_{E/F}$ fits into a short exact sequence
$$0\rightarrow E^{\times} \rightarrow W_{E/F}\rightarrow \Delta\rightarrow 0,$$
corresponding to the fundamental class $[u_{E/F}]\in H^2(\Delta, E^{\times})$ by \cite[(1.2)]{tate}. Let $\widehat{W}_{E/F}$ be the profinite completion of 
$W_{E/F}$. The Artin map $\Art_E: E^{\times}\rightarrow \Gamma_E^{\ab}$ 
induces an isomorphism between the profinite completion of $E^{\times}$ and 
$\Gamma_E^{\ab}$. This yields a natural isomorphism 
\begin{equation}\label{semi_artin}
\widehat{W}_{E/F} \cong \Gamma_F/ [\Gamma_E, \Gamma_E],
\end{equation} 
where $[\Gamma_E, \Gamma_E]$ is the closure of the subgroup generated by the commutators in $\Gamma_E$.

\begin{lem}\label{factor_one} Every continuous representation $\rho: \Gamma_F \rightarrow G(A)$ satisfying $\Pi\circ \rho= \pi$, 
where $A$ is a finite discrete $\OO$-algebra, 
factors through the quotient $\Gamma_F\twoheadrightarrow \widehat{W}_{E/F}$. 
\end{lem} 
\begin{proof} Since $\rho(\Gamma_E) \subseteq G^0(A)$ and $G^0(A)$ is commutative and Hausdorff the assertion follows from \eqref{semi_artin}.
\end{proof} 

\begin{lem}\label{factor_two} Let $A\in R^{\ps}_{\Thetabar}\text{-}\alg$. Then every 
$\rho\in X^{\gen}_{G, \Thetabar}(A)$ factors through the quotient 
$\Gamma_F\twoheadrightarrow \widehat{W}_{E/F}$.
\end{lem}

\begin{proof} The restriction of $\rhobar$ to $\Gamma_E$ takes values in $G^0(k)$. 
Let $\Psibar$ be the $G^0$-pseu\-do\-character of $\rhobar|_{\Gamma_E}$ and let $\Psi^u\in D^{\ps}_{\Psibar}(R^{\ps}_{\Psibar})$ 
be the universal deformation of $\Psibar$.  
Since $G^0$ is commutative, Lemma \ref{pc_is_rep} implies that there is a continuous group homomorphism 
$\psi^u: \Gamma_E \rightarrow G^0(R^{\ps}_{\Psibar})$ such that $\Psi^u = \Theta_{\psi^u}$. In particular, 
\begin{equation}\label{commu}
\psi^u(\gamma)=1, \quad  \forall \gamma \in [\Gamma_E, \Gamma_E].
\end{equation}

Since $\Theta^u|_{\Gamma_E}$ is a deformation of $\Psibar$ to $R^{\ps}_{\Thetabar}$ there is a homomorphism
$R^{\ps}_{\Psibar}\rightarrow R^{\ps}_{\Thetabar}$ such that
\begin{equation}\label{eq_univ}
\Theta^u|_{\Gamma_E} = \Psi^u\otimes_{R^{\ps}_{\Psibar}} R^{\ps}_{\Thetabar}.
\end{equation}
If $\rho\in X^{\gen}_{G, \Thetabar}(A)$ then $\Theta_{\rho}= \Theta^u\otimes_{R^{\ps}_{\Thetabar}} A$, and it follows from 
\eqref{eq_univ} that $\rho|_{\Gamma_E}= \psi^u \otimes_{R^{\ps}_{\Psibar}} A$. The assertion follows from \eqref{commu}.
\end{proof}

\begin{thm}\label{main} 
    There is a finite extension $L'$ of $L$ with the ring of integers $\OO'$, such that the following hold: 
    \begin{enumerate}
    \item $G^0_{\OO'}$ is split and $(G/G^0)_{\OO'}$ 
    is a constant group scheme;
    \item $\Rep^{\Gamma_F}_{G, \pi}(\OO')$ is non-empty.
    \end{enumerate} 
    Moreover, if $(1)$ and $(2)$ hold then there are isomorphisms of $\OO'$-algebras: 
    \begin{itemize}
    \item[(3)] $R^{\ps}_{\Thetabar}\otimes_{\OO} \OO' \cong \OO'\br{(\Gamma_E^{\ab,p}\otimes M)^{\Delta}}\cong \OO'[(\mu_{p^{\infty}}(E)\otimes M)^{\Delta}]\br{x_1, \ldots, x_r}$;
    \item[(4)] $A^{\gen}_{G, \Thetabar}\otimes_{\OO} \OO'\cong R^{\ps}_{\Thetabar}\otimes_{\OO} \OO'[t_1^{\pm 1}, \ldots, t_s^{\pm 1}]$;
    \item[(5)] $R^{\square}_{\rhobar}\otimes_{\OO} \OO'\cong 
    (R^{\ps}_{\Thetabar}\otimes_{\OO} \OO')\br{t_1, \ldots, t_s}\cong \OO'[(\mu_{p^{\infty}}(E)\otimes M)^{\Delta}]\br{z_1, \ldots, z_{r+s}}$;
    \end{itemize}
    where $M$ is the character lattice of $G^0_{\OO'}$, $r= \rank_{\ZZ} M \cdot [F:\Qp] +\rank_{\ZZ} M_{\Delta}$, $s=\rank_\ZZ M -\rank_{\ZZ} M_{\Delta}$.
\end{thm}

\begin{proof} 
If $L'$ is a finite extension of $L$ with the ring of integers $\OO'$ and residue field $k'$ then the 
functor $D^{\ps}_{\Thetabar_{k'}}: \Aa_{\OO'}\rightarrow \Set$ is 
pro-representable by $R^{\ps}_{\Thetabar}\otimes_{\OO} \OO'$ and 
analogous statements hold for $X^{\gen}_{G, \Thetabar_{k'}}$ and $D^{\square}_{\rhobar_{k'}}$. Thus it is enough to prove the statement 
for $\OO'=\OO$, after replacing $L$ by a finite extension. 

Since $G$ is a generalised torus, after replacing $L$ by a finite unramified extension we may assume that $G^0$ is split and $G/G^0$ is
a constant group scheme. We may assume that  $G/G^0=\underline{\Delta}$ as replacing $G$ by the preimage of $\underline{\Delta}$
does not change the functors under consideration. The character lattice $M$ of $G^0$ does not change if we further replace $L$ by a finite 
extension. As explained at the beginning of \Cref{sec_adm_rep} we have an action of $\Delta$ on $M$.

We will apply the results of previous sections with  $\Gamma_1=W_{E/F}$ and $\Gamma_2=E^{\times}$. It is a fundamental result of Langlands proved in \cite{langlands_tori} (see also a nice exposition 
by Birkbeck \cite[Proposition 2.0.3]{birkbeck}) that there are natural isomorphisms:
\begin{equation}\label{fundamental}
H_1(W_{E/F}, M) \cong H_1(E^{\times}, M)^{\Delta} \cong (E^{\times}\otimes M)^{\Delta}.
\end{equation}
For each $c\in \Delta$ we choose a coset representative $\bar{c}\in W_{E/F}$. We have constructed 
a $\Delta$-action on $\mathcal E\colonequals E^{\times}\oplus I_{\Delta}$, which depends on this choice,  such that we have an extension of $\ZZ[\Delta]$-modules
$$0\rightarrow E^{\times} \rightarrow \mathcal E \rightarrow I_{\Delta}\rightarrow 0$$
and the image of the extension class in $H^2(\Delta, E^{\times})$ under the isomorphisms
$$\Ext^1_{\ZZ[\Delta]}(I_{\Delta}, E^{\times})\cong \Ext^2_{\ZZ[\Delta]}(\ZZ, E^{\times})\cong H^2(\Delta, E^{\times})$$ is equal to $[u_{E/F}]$. 

 Let $N$ be the pro-$p$ completion of $(\mathcal E\otimes M)_{\Delta}$.
 It follows from Proposition \ref{main_rank} and Euler--Poincar\'e characteristic formula that 
 \begin{equation}\label{eq_EP}
 h^1_{\cont}(\Gamma_F, \ad\rhobar)-h^0_{\cont}(\Gamma_F, \ad \rhobar)-
 h^2_{\cont}(\Gamma_F, \ad \rhobar) =\rank_{\Zp} N -\dim G_k.
 \end{equation}
Lemma \ref{factor_one} and \eqref{eq_EP} ensure that the assumptions of Proposition \ref{lci} are satisfied, 
and hence after replacing $L$ by a finite extension we may ensure that $\Rep^{\widehat{W}_{E/F}}_{G, \pi}(\OO)$ is non-empty
 by Corollary \ref{7_7}. In particular, $\Rep^{\Gamma_F}_{G,\pi}(\OO)$ and 
$\Rep^{W_{E/F}}_{G, \pi}(\OO)$ are also non-empty. 

Lemma \ref{comes_from_rep} implies that there is $\rho\in X^{\gen}_{G, \rhobarss}(R^{\ps}_{\Thetabar})$ such that $\Theta^u = \Theta_{\rho}$. 
Since $\rho$ factors through the quotient $\Gamma_F \twoheadrightarrow \widehat{W}_{E/F}$ by Lemma \ref{factor_two} 
we conclude that $\Theta^u$ is obtained from  a $G$-pseudocharacter of $\widehat{W}_{E/F}$ via inflation to $\Gamma_F$. 
This together with Lemmas \ref{factor_one} and \ref{factor_two} imply that in the definitions of the functors 
$D^{\ps}_{\Thetabar}$, $D^{\square}_{\rhobar}$ and $X^{\gen}_{G, \Thetabar}$ we can replace $\Gamma_F$ with $\widehat{W}_{E/F}$
without changing the functors themselves.  

Lemma \ref{no_name} and \eqref{fundamental} imply that $R^{\ps}_{\Thetabar}$ is isomorphic to $\OO\br{N}$, where $N$ is the pro-$p$ completion 
of $(E^{\times}\otimes M)^{\Delta}$. It follows from Lemma \ref{duck_duck} 
that $N\cong (\Gamma^{\ab, p}_E\otimes M)^{\Delta}$, 
which is a finitely generated $\Zp$-module of rank $r$ and torsion subgroup isomorphic to $(\mu_{p^{\infty}}(E)\otimes M)^{\Delta}$ by \Cref{rank_tors}.
This yields the isomorphisms in part (3). Part (4) follows from Proposition \ref{paris}. 

 The map $\rho\mapsto \Theta_{\rho}$ 
induces a map of local $\OO$-algebras $R^{\ps}_{\Thetabar}\rightarrow R^{\square}_{\rhobar}$
and hence a map of $R^{\ps}_{\Thetabar}$-algebras
$A^{\gen}_{G, \Thetabar}\rightarrow R^{\square}_{\rhobar}$. 
Since $E^{\times}/ (E^{\times})^p$ is finite, Lemma \ref{completion_def_square} implies that 
$R^{\square}_{\rhobar}$ is the completion of 
$\OO(\Rep^{W_{E/F}}_{G,\pi})$ with respect to the maximal ideal corresponding to $\rhobar$.
By considering the  composition $\OO(\Rep^{W_{E/F}}_{G,\pi}) \rightarrow A^{\gen}_{G, \Thetabar}\rightarrow R^{\square}_{\rhobar}$ we deduce that  
that it induces a natural isomorphism between $R^{\square}_{\rhobar}$ and the  completion of $A^{\gen}_{G, \overline{\Theta}}$ with respect to the maximal ideal corresponding to $\rhobar$. Part (5) then follows
from parts (3) and (4).
%
\end{proof}

We will now deduce some corollaries, which hold without extending the scalars.

\begin{cor}\label{form_smooth} The map $R^{\ps}_{\Thetabar} \rightarrow R^{\square}_{\rhobar}$ 
is formally smooth. In particular, it is flat and induces a bijection 
between the sets of irreducible components. 
\end{cor} 

\begin{proof} Since the map $R^{\ps}_{\Thetabar}\rightarrow \OO'\otimes_{\OO}R^{\ps}_{\Thetabar}$ is faithfully flat, 
the assertion follows from \cite[\href{https://stacks.math.columbia.edu/tag/06CM}{Tag 06CM}]{stacks-project} and part (5) of Theorem \ref{main}.
\end{proof}

\begin{cor} The map $R^{\ps}_{\Thetabar}\rightarrow A^{\gen}_{G, \Thetabar}$
is smooth and induces a bijection between the sets of irreducible components.
Moreover, $A^{\gen}_{G, \Thetabar}$
is flat over $\OO$ of relative dimension $\dim G_k \cdot ([F:\Qp]+1)$.
\end{cor} 

\begin{proof} The map $R^{\ps}_{\Thetabar}\rightarrow A^{\gen}_{G, \Thetabar}$
is formally smooth by the same argument as in Corollary \ref{form_smooth} using 
part (4) of Theorem \ref{main}. Since it is of finite type it is smooth by 
\cite[\href{https://stacks.math.columbia.edu/tag/00TN}{Tag 00TN}]{stacks-project}.
The map $X^{\gen}_{G, \Thetabar}\rightarrow X^{\ps}_{\Thetabar}:=\Spec R^{\ps}_{\Thetabar}$ is flat and of finite presentation and hence 
open by \cite[\href{https://stacks.math.columbia.edu/tag/00I1}{Tag 00I1}]{stacks-project}. By \cite[\href{https://stacks.math.columbia.edu/tag/004Z}{Tag 004Z}]{stacks-project} it is enough to show the fibres of $X^{\gen}_{G, \Thetabar} \rightarrow X^{\ps}_{\Thetabar}$ are irreducible. 

To ease the notation we let $R=R^{\ps}_{\Thetabar}$, $A=A^{\gen}_{G, \Thetabar}$, $R'=R\otimes_{\OO} \OO'$ and $A'=A\otimes_{\OO} \OO'$.
Let $x:\Spec \kappa \rightarrow \Spec R$ be a geometric point and 
let $x': \Spec \kappa\rightarrow \Spec R'$ be a point above $x$. The map 
$$ A\otimes_{R, x} \kappa \rightarrow A'\otimes_{R', x'}\kappa$$  
is an isomorphism. Part (2) of Theorem \ref{main} implies that 
$$A'\otimes_{R', x'}\kappa\cong \kappa[t_1^{\pm 1}, \ldots, t_s^{\pm 1} ]$$ 
Thus the fibres of $X^{\gen}_{G, \Thetabar} \rightarrow X^{\ps}_{\Thetabar}$ are irreducible and the result follows. 

The last assertion follows from the fact that $\OO'$ is finite and free over $\OO$, parts (3) and (4) of Theorem \ref{main}
and $\dim G_k = \rank_{\ZZ} M$. 
\end{proof} 

\begin{lem}\label{defG2} Let $A\in R^{\ps}_{\Thetabar}\text{-}\alg$, let $\rho\in X^{\gen}_{G,\Thetabar}(A)$ and let $\tau: G\hookrightarrow \mathbb A^n$ 
be a closed immersion of $\OO$-schemes. Then $\tau(\rho(\Gamma_F))$ 
is contained in a finitely generated $R^{\ps}_{\Thetabar}$-submodule 
of $A^n=\mathbb A^n(A)$.
\end{lem}

\begin{proof} It is enough to verify the assertion after extending the 
scalars to $\OO'$ given by Theorem \ref{main}. This follows from Lemma 
\ref{factor_two} and Lemma \ref{defG2}.
\end{proof}

\begin{remar} Lemma \ref{defG2} implies that the scheme $X^{\gen}_{G, \Thetabar}$ 
coincides with the scheme $X^{\gen}_{G, \rhobarss}$ defined in \cite{defG}
for a generalised reductive group $G$, see \cite[Proposition 7.3]{defG}.
\end{remar}

\begin{cor}\label{components2} 
    Let $p^m$ be the order of $(\mu_{p^{\infty}}(E)\otimes M)^{\Delta}$. 
    Assume that $\OO$ contains all the $p^m$-th roots of unity and (1) and (2) 
    in Theorem \ref{main} hold with $\OO'=\OO$. 
    Then there is a canonical  action of the character group 
    $\mathrm X((\mu_{p^{\infty}}(E)\otimes M)^{\Delta})$ on the set of irreducible 
    components of $\Spec R^{\ps}_{\Thetabar}$, $\Spec R^{\square}_{\rhobar}$ and $X^{\gen}_{G, \Thetabar}$, respectively. Moreover, the following hold:
    \begin{enumerate}
    \item this action is faithful and transitive;
    \item irreducible components of $\Spec R^{\square}_{\rhobar}$ and 
    $\Spec R^{\ps}_{\Thetabar}$ are formally smooth over $\OO$;
    \item irreducible components of $X^{\gen}_{G, \Thetabar}$ are 
    of the form $$\Spec \OO\br{x_1, \ldots, x_r}[t_1^{\pm 1}, \ldots, t_s^{\pm 1}],$$ 
    where $r$ and $s$ are as in Theorem \ref{main}.
    \end{enumerate}
\end{cor}

\begin{proof} The assertion  follows from \Cref{components1}.
\end{proof}

\begin{cor}
    Let $\varphi : G \to H$ be a surjection of generalised tori over $\OO$ and let $\rhobar : \Gamma_F \to G(k)$ be a continuous representation.
    Then the map $(R^{\square}_{\varphi\circ \rhobar}/\varpi)^{\red}\rightarrow (R^{\square}_{\rhobar}/\varpi)^{\red}$
    is flat, where $\red$ indicates reduced rings, 
    and the fibre at the closed point has dimension $([F:\Qp]+1)(\dim G_k - \dim H_k)$.
\end{cor}

\begin{proof} 
    If $\mathcal P$ is a finitely generated $\Zp$-module and $R=k\br{\mathcal P}$ is the completed group algebra
    of $\mathcal P$ then $R^{\red}\cong k\br{\mathcal P^{\tf}}$, where $\mathcal P^{\tf}$ is the maximal torsion free
    quotient of $\mathcal P$. In particular, $R^{\red}\cong k\br{x_1,\ldots, x_d}$, where $d= \dim_{\Qp} (\mathcal P\otimes_{\Zp} \Qp)$.
    
    Let  $f:\mathcal P_1\rightarrow \mathcal P_2$ be a homomorphism of 
    finitely generated $\Zp$-modules and let $R_i=k\br{\mathcal P_i}$.  
    If $\ker f$ is torsion then $\mathcal P_1^{\tf}$ is a submodule of 
    $\mathcal P_2^{\tf}$ and the fibre $k\otimes_{R_1^{\red}} R_2^{\red}$ 
    is isomorphic to $k\br{\mathcal P^{\tf}_2/\mathcal P^{\tf}_1}$.
    Since the dimension of the fibre is equal to $d_2-d_1$, where
    $d_i= \dim_{\Qp} (\mathcal P_i\otimes_{\Zp} \Qp)$, we deduce 
    from \cite[\href{https://stacks.math.columbia.edu/tag/00R4}{Tag 00R4}]{stacks-project} that $R_1^{\red}\rightarrow R_2^{\red}$
    is flat.

    Let $M$ be the character lattice 
    of $G^0$ and let $N$ be the character lattice 
    of $H^0$. Then $N\subseteq M$ and it follows from
    \eqref{long_exact} and \Cref{tate_finite} that the kernel of the map $((\mathcal E \otimes N)^{\wdp})_{\Delta}\rightarrow ((\mathcal E \otimes M)^{\wdp})_{\Delta}$
    is torsion. Since $R_{\rhobar}^{\square}/\varpi\cong 
    k\br{((\mathcal E \otimes M)_{\Delta})^{\wdp}}$ by Lemma \ref{over_k}
    and we may swap pro-$p$ completion with $\Delta$-coinvariants, we 
    obtain the assertion by letting 
    $\mathcal P_1=((\mathcal E \otimes N)_{\Delta})^{\wdp}$ and 
    $\mathcal P_2=((\mathcal E \otimes M)_{\Delta})^{\wdp}$. The assertion about the dimension of the fibre follows from \Cref{main_rank}.
\end{proof}

\bibliographystyle{plain}
\bibliography{Ref}

\end{document}